\documentclass[english]{article}

\usepackage[T1]{fontenc}
\usepackage[latin9]{inputenc}
\usepackage{geometry}
\usepackage{color}
\usepackage{babel}
\usepackage{refstyle}
\usepackage{float}
\usepackage{fancybox}
\usepackage{calc}
\usepackage{amsmath}
\usepackage{amsthm}
\usepackage{amssymb}
\usepackage{graphicx}
\usepackage{wasysym}
\usepackage[all]{xy}
\PassOptionsToPackage{normalem}{ulem}
\usepackage{ulem}
\usepackage[unicode=true,
 bookmarks=true,bookmarksnumbered=false,bookmarksopen=false,
 breaklinks=false,pdfborder={0 0 1},backref=false,colorlinks=true]
 {hyperref}
\hypersetup{pdftitle={The Conservative Matrix Field},
 pdfauthor={Ofir David},
 pdfkeywords={generalized continued fraction, irrationality}}

\makeatletter

\AtBeginDocument{\providecommand\secref[1]{\ref{sec:#1}}}
\AtBeginDocument{\providecommand\thmref[1]{\ref{thm:#1}}}
\AtBeginDocument{\providecommand\subsecref[1]{\ref{subsec:#1}}}
\AtBeginDocument{\providecommand\claimref[1]{\ref{claim:#1}}}
\AtBeginDocument{\providecommand\lemref[1]{\ref{lem:#1}}}
\AtBeginDocument{\providecommand\exaref[1]{\ref{exa:#1}}}
\AtBeginDocument{\providecommand\enuref[1]{\ref{enu:#1}}}
\AtBeginDocument{\providecommand\corref[1]{\ref{cor:#1}}}
\AtBeginDocument{\providecommand\defref[1]{\ref{def:#1}}}
\AtBeginDocument{\providecommand\remref[1]{\ref{rem:#1}}}
\AtBeginDocument{\providecommand\eqref[1]{\ref{eq:#1}}}
\AtBeginDocument{\providecommand\factref[1]{\ref{fact:#1}}}
\AtBeginDocument{\providecommand\appref[1]{\ref{app:#1}}}
\RS@ifundefined{subsecref}
  {\newref{subsec}{name = \RSsectxt}}
  {}
\RS@ifundefined{thmref}
  {\def\RSthmtxt{theorem~}\newref{thm}{name = \RSthmtxt}}
  {}
\RS@ifundefined{lemref}
  {\def\RSlemtxt{lemma~}\newref{lem}{name = \RSlemtxt}}
  {}

\theoremstyle{plain}
\newtheorem{thm}{\protect\theoremname}
\theoremstyle{remark}
\newtheorem{rem}[thm]{\protect\remarkname}
\theoremstyle{definition}
\newtheorem{defn}[thm]{\protect\definitionname}
\theoremstyle{definition}
\newtheorem{example}[thm]{\protect\examplename}
\theoremstyle{plain}
\newtheorem{lem}[thm]{\protect\lemmaname}
\theoremstyle{plain}
\newtheorem{claim}[thm]{\protect\claimname}
\theoremstyle{definition}
\newtheorem{notation}[thm]{\protect\notationname}
\theoremstyle{plain}
\newtheorem{fact}[thm]{\protect\factname}
\theoremstyle{plain}
\newtheorem{cor}[thm]{\protect\corollaryname}

\@ifundefined{date}{}{\date{}}
\hypersetup{pdfstartview=}
\newref{def}{name=Definition~}
\newref{ex}{name=Example~}
\newref{eq}{name=Equation~}
\newref{exa}{name=Example~}
\newref{sub}{name=Subsection~}
\newref{cor}{name=Corollary~}
\newref{rem}{name=Remark~}
\newref{prop}{name=Proposition~}
\newref{lem}{name=Lemma~}
\newref{thm}{name=Theorem~}
\newref{sec}{name=Section~}
\newref{claim}{name=Claim~}
\newref{app}{name=Appendix~}

\hypersetup{citecolor=blue}

\newcommand{\xyR}[1]{
  \xydef@\xymatrixrowsep@{#1}}
\newcommand{\xyC}[1]{
  \xydef@\xymatrixcolsep@{#1}}

\makeatother

\providecommand{\claimname}{Claim}
\providecommand{\corollaryname}{Corollary}
\providecommand{\definitionname}{Definition}
\providecommand{\examplename}{Example}
\providecommand{\factname}{Fact}
\providecommand{\lemmaname}{Lemma}
\providecommand{\notationname}{Notation}
\providecommand{\remarkname}{Remark}
\providecommand{\theoremname}{Theorem}

\begin{document}
\global\long\def\call{\mathcal{L}}%
\global\long\def\nn{\mathcal{N}}%
\global\long\def\ff{\mathcal{F}}%
\global\long\def\aa{\mathcal{A}}%
\global\long\def\RR{\mathbb{R}}%
\global\long\def\EE{\mathbb{E}}%
\global\long\def\CC{\mathbb{C}}%
\global\long\def\QQ{\mathbb{Q}}%
\global\long\def\ZZ{\mathbb{Z}}%
\global\long\def\NN{\mathbb{N}}%
\global\long\def\KK{\mathbb{K}}%
\global\long\def\SL{\mathrm{SL}}%
\global\long\def\GL{\mathrm{GL}}%
\global\long\def\ds{\mathrm{ds}}%
\global\long\def\dnu{\mathrm{d\nu}}%
\global\long\def\dmu{\mathrm{d\mu}}%
\global\long\def\dt{\mathrm{dt}}%
\global\long\def\dw{\mathrm{dw}}%
\global\long\def\dx{\mathrm{dx}}%
\global\long\def\dy{\mathrm{dy}}%
\global\long\def\norm#1{\left\Vert #1\right\Vert }%
\global\long\def\limfi#1{{\displaystyle \lim_{#1\to\infty}}}%
\global\long\def\arrfi#1{\overset{#1\to\infty}{\longrightarrow}}%
\global\long\def\flr#1{\left\lfloor #1\right\rfloor }%
\global\long\def\lcm{\mathrm{lcm}}%
\global\long\def\fbf{\left(f\bar{f}\right)}%

\title{\textbf{\huge{}The conservative matrix field}}
\author{Ofir David\\
The Ramanujan Machine team, Faculty of Electrical and Computer Engineering,\\
Technion - Israel Institute of Technology, Haifa 3200003, Israel}
\maketitle
\begin{abstract}
We present a new structure called the \textquotedbl conservative
matrix field,\textquotedbl{} initially developed to elucidate and
provide insight into the methodologies employed by Ap\'ery in his
proof of the irrationality of $\zeta\left(3\right)$. This framework
is also applicable to other well known mathematical constants, including
$e,\;\pi,\;\ln\left(2\right)$, and more, and can be used to study
their properties. Moreover, the conservative matrix field exhibits
inherent connections to various ideas and techniques in number theory,
thereby indicating promising avenues for further applications and
investigations.
\end{abstract}

\section{Introduction}

The Riemann zeta function $\zeta\left(s\right)$ is a complex valued
function that plays a crucial role in mathematics. It is defined as
$\zeta\left(s\right)=\sum_{n=1}^{\infty}\frac{1}{n^{s}}$ for complex
numbers $s$ with $Re\left(s\right)>1$ and can be extended analytically
to all of $\CC$ with a simple pole at $s=1$. In particular, the
values $\zeta\left(d\right)$ for integers $d\geq2$ have significant
implications in a number of areas, e.g. $\frac{1}{\zeta\left(d\right)}$
is the probability of choosing a random integer which is not divisible
by $m^{d}$ for some integer $m\geq1$. While the even evaluations
$\zeta\left(2d\right)$ are well understood, with $\zeta\left(2\right)=\frac{\pi^{2}}{6}$
and more generally $\frac{\zeta\left(2d\right)}{\pi^{2d}}$ are rational
numbers, the behavior of odd evaluations $\zeta\left(2d+1\right)$
remains largely unknown.

One of the main results about these odd evaluations was in 1978 where
Ap\'ery showed that $\zeta\left(3\right)$ is irrational \cite{apery_irrationalite_1979}.
As Ap\'ery's proof was complicated, subsequent attempts were made
to explain and simplify it, see for example van der Poorten \cite{van_der_poorten_proof_1979},
and others tried to reprove it all together, e.g. Beukers in \cite{beukers_note_2013}.
Further research \cite{rivoal_fonction_2000,zudilin_one_2001} has
also shown that there are infinitely many odd integers for which $\zeta\left(2d+1\right)$
is irrational, and in particular at least one of $\zeta\left(5\right),\zeta\left(7\right),\zeta\left(9\right)$,
and $\zeta\left(11\right)$ is irrational.

One of the main goals of this paper is to provide some motivation
to the steps used in Ap\'ery's proof through the creation of a novel
mathematical structure referred to as the \textbf{conservative matrix
field}. This structure will allow us to understand Ap\'ery's original
proof and provide a framework for studying other natural constants
such as $\zeta\left(2\right),\pi$, and $e$, with the potential to
uncover new relationships and properties among them. Moreover, while
we will mainly work over the integers, this structure seems to have
natural generalizations for general metric fields.

\medskip{}

The conservative matrix field structure is based on \textbf{generalized
continued fractions}, which are number presentations of the form
\[
\KK_{1}^{\infty}\frac{b_{k}}{a_{k}}:=\frac{b_{1}}{a_{1}+\frac{b_{2}}{a_{2}+\frac{b_{3}}{\overset{a_{3}+}{\phantom{a}}\ddots}}}=\limfi n\frac{b_{1}}{a_{1}+\frac{b_{2}}{a_{2}+\frac{b_{3}}{\ddots+\frac{b_{n}}{a_{n}+0}}}}\qquad a_{i},b_{i}\in\CC,
\]
where the numbers in the sequence on the right for which we look for
the limit are called the \textbf{convergents} of that expansion.

Their much more well known cousins, the \textbf{simple continued fractions}
where $b_{k}=1$ and $a_{k}\geq1$ are integers, have been studied
extensively and are connected to many research areas in mathematics
and in general. In particular, the original goal of these continued
fractions was to find the ``best'' rational approximations for a
given irrational number, which are given by the convergents defined
above.

While irrational numbers have a unique simple continued fraction expansion,
and rationals have two expansions, there can be many presentations
in the generalized version (more details in \secref{Generalized-continued-fractions}).
The uniqueness in the simple continued fraction expansion allows us
to extract a lot of information from the expansion, and while we lose
this property in the generalized form, what we gain is the option
to find ``nice'' generalized continued fractions which are easier
to work with. In particular, we are interested in \textbf{polynomial
continued fractions} where $a_{k}=a\left(k\right),\;b_{k}=b\left(k\right)$
with $a,b\in\ZZ\left[x\right]$.

For example, in the $\zeta\left(3\right)$ case, the simple continued
fraction is 
\[
\zeta\left(3\right)=[1;4,1,18,1,1,1,4,1,9,...]=1+\frac{1}{4+\frac{1}{1+\frac{1}{18+\frac{1}{1+\ddots}}}},
\]
where the coefficients $1,4,1,18,1,...$ don't seem to have any usable
pattern. However, it has a much simpler generalized continued fraction
form
\[
\zeta\left(3\right)=\frac{1}{1+\KK_{1}^{\infty}\frac{-i^{6}}{i^{3}+\left(1+i\right)^{3}}}=\frac{1}{1-\frac{1^{6}}{1^{3}+2^{3}-\frac{2^{6}}{2^{3}+3^{3}-\frac{3^{6}}{3^{3}+4^{3}-{\scriptscriptstyle \ddots}}}}},
\]
where the convergents in this expansion are the standard approximations
$\sum_{1}^{n}\frac{1}{k^{3}}$ for $\zeta\left(3\right)$. Moreover,
this abundance of presentations allows us to find many presentations
for $\zeta\left(3\right)$ which can be combined together to find
a ``good enough'' presentation where the convergents converge fast
enough to prove that $\zeta\left(3\right)$ is irrational. In particular,
in Ap\'ery's original proof, and in ours, we eventually show that
\[
\zeta\left(3\right)=\frac{6}{5+\KK_{1}^{\infty}\frac{-k^{6}}{17\left(k^{3}+\left(1+k\right)^{3}\right)-12\left(k+\left(1+k\right)\right)}}.
\]
\smallskip{}

The irrationality proof uses a very elementary argument (see \secref{Irrationality-testing})
that shows that if $\frac{p_{n}}{q_{n}}\to L$ where $\frac{p_{n}}{q_{n}}$
are reduced rational numbers with $\left|q_{n}\right|\to\infty$,
and $\left|L-\frac{p_{n}}{q_{n}}\right|=o\left(\frac{1}{\left|q_{n}\right|}\right)$,
then $L$ must be irrational. Moreover, we can measure how irrational
$L$ is by looking for $\delta>0$ such that $\left|L-\frac{p_{n}}{q_{n}}\right|\sim\frac{1}{\left|q_{n}\right|^{1+\delta}}$.
The main object of this study, the \textbf{conservative matrix field}
defined in \secref{Definition-properties-CMF},\textbf{ }is an algebraic
object that collects infinitely many related such approximations $\frac{p_{n,m}}{q_{n,m}}$
arranged on the integer lattice in the positive quadrant. Computing
$\delta_{n,m}$ for each approximation in the $\zeta\left(3\right)$
matrix field, namely $\delta_{n,m}=-1-\frac{\ln\left|L-\frac{p_{n,m}}{q_{n,m}}\right|}{\ln\left|q_{n,m}\right|}$
where the rational $\frac{p_{n,m}}{q_{n,m}}$ is reduced, and plotting
them as a heat map we get the following

\begin{figure}[H]
\begin{centering}
\includegraphics[scale=0.25]{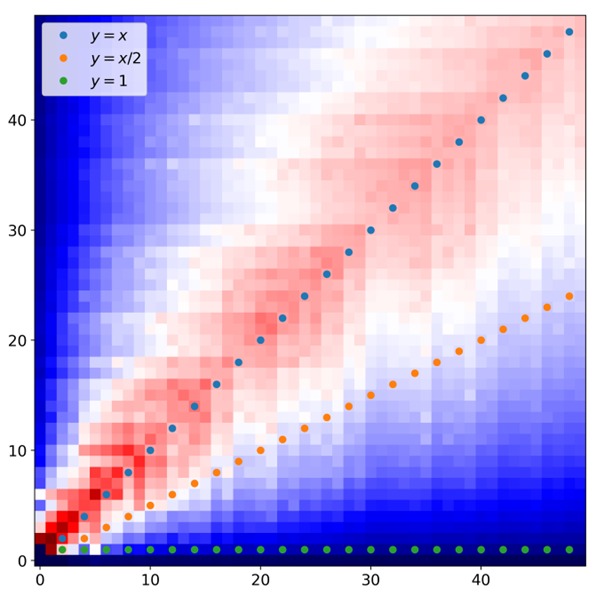}
\par\end{centering}
\caption{(Figure by Rotem Elimelech) The gradient color from red$\to$white$\to$blue
correspond to $\delta_{n,m}=-1-\log_{q_{n,m}}\left|L-\frac{p_{n,m}}{q_{n,m}}\right|$
going from positive$\to$zero$\to$negative.}
\end{figure}

As we shall see, the $X$-axis and $Y$-axis correspond more or less
to the standard approximations of $\zeta\left(3\right)$, namely $\sum_{1}^{n}\frac{1}{k^{3}}$,
which do not converge fast enough to show irrationality, while on
the diagonal we get the expansion mentioned above used by Ap\'ery
to prove the irrationality.

\medskip{}

The conservative matrix field structure is not only a way to understand
Ap\'ery's original proof, but seems to have a much broader range
of applications. There are many places that study generalized continued
fractions and in particular polynomial continued fractions (see for
example \cite{apostol_introduction_1998,jones_continued_1980,pincherle_delle_1894,laughlin_real_2004})
. This paper originated in the Ramanujan machine project \cite{raayoni_generating_2021}
which aimed to find polynomial continued fraction presentations to
interesting mathematical constants using computer automation. With
the goal of trying to prove many of the conjectures discovered by
the computer, and along the way understand Ap\'ery's proof, this
conservative matrix field structure was found. These computer conjectures
suggest that there is still much to be explored in this field and
that this new structure is just a step towards a deeper understanding
of mathematical constants and their relations.

\subsection{Structure of the paper}

We present an overview of the paper's structure to delineate the development
of the conservative matrix field. The construction of the conservative
matrix field itself and its properties begins in \secref{Definition-properties-CMF},
while the previous sections are more of an introduction to the subject
of polynomial continued fractions.

\textbf{Irrationality Testing (}\secref{Irrationality-testing}):
We begin by exploring the concept of good rational approximations
and their role in irrationality testing. These approximations are
closely linked to simple continued fractions, but also still close
to other continued fraction versions.

\textbf{The polynomial continued fractions (}\secref{Generalized-continued-fractions}):
We provide definitions for various types of continued fractions, with
a specific focus on polynomial continued fractions, and their role
in irrationality testing. In particular, we recall in \thmref{recurrence-roots}
the Euler continued fractions from \cite{david_euler_2023} which
serve as the fundamental building blocks of the conservative matrix
field, where by combining infinitely many such continued fractions,
we generate new rational approximations.

\textbf{The conservative matrix field }(\secref{Definition-properties-CMF}):
After defining the conservative matrix field, we present an intriguing
matrix field construction in \subsecref{matrix-field-construction}
with several interesting properties, where in particular each such
matrix field is accompanied by a dual matrix field, elaborated upon
in \subsecref{The-dual-matrix-field}.

\textbf{The $\zeta\left(3\right)$ matrix field} (\secref{The-z3-case}):
In this section, we explore $\zeta\left(3\right)$ matrix field's
structure, its intriguing properties, and its connections to the irrationality
proof of $\zeta\left(3\right)$. The natural questions arising from
this study often parallel ideas presented in Ap\'ery's proof.

\textbf{Future Directions and Connections} (\secref{On-future-fractions}):
As a conclusion, we propose potential avenues for further development
of the conservative matrix field, and its potential connections to
other fields within number theory, hinting at exciting possibilities
for future research.

\newpage{}

\section{\label{sec:Irrationality-testing}Rational Approximations and Irrationality
Testing}

The simple continued fractions were first introduced in order to find
good rational approximations for a given number, and in a sense also
the best rational approximations. Interestingly enough, the numbers
which have ``very good'' rational approximations are exactly the
irrational numbers. This idea of proving irrationality through good
rational approximations is quite general and will be one of the main
applications to polynomial continued fractions, and the conservative
matrix field which we introduce in this paper. As such, we begin with
this irrationality test.

\medskip{}

Any real number $L$ can be approximated by rational numbers. More
over, for any denominator $q\in\NN$ there is a numerator $p\in\ZZ$
such that $\left|L-\frac{p}{q}\right|\leq\frac{1}{q}$. Looking for
approximations where the error is much smaller than $\frac{1}{q}$
is the starting point of the study of Diophantine approximations.
One of the basic results in this field is Dirichlet's theorem, which
states that if we can choose the denominator intelligently, then there
are approximations with much smaller error.
\begin{thm}[\textbf{Dirichlet's theorem for Diophantine approximation}]
\label{thm:Dirichlet}Given any number $L\in\RR$, there are infinitely
many $p_{n},q_{n}\in\ZZ$ with $q_{n}\to\infty$ such that 
\[
\left|L-\frac{p_{n}}{q_{n}}\right|\leq\frac{1}{q_{n}^{2}}.
\]
\end{thm}

\begin{rem}
One of the important properties of simple continued fractions, is
that they can generate such approximations as in Dirichlet's theorem.
\end{rem}

Dirichlet's theorem as written above is trivial when $L=\frac{p}{q}$
is rational by simply taking $\left(p_{n},q_{n}\right)=\left(np,nq\right)$.
However if we also require that the $\frac{p_{n}}{q_{n}}$ are distinct,
then this theorem is no longer trivial, or even true. In this case,
once $L=\frac{p}{q}\neq\frac{p_{n}}{q_{n}}$ we get that 
\[
\left|L-\frac{p_{n}}{q_{n}}\right|=\left|\frac{p\cdot q_{n}-q\cdot p_{n}}{q\cdot q_{n}}\right|\geq\frac{1}{\left|q\cdot q_{n}\right|},
\]
where the emphasize is that $q$ is constant, so that $\left|q_{n}L-p_{n}\right|\geq\frac{1}{q}$
is bounded from below. When $L$ is irrational, it is easy to see
that the approximations $\frac{p_{n}}{q_{n}}$ cannot become constant
at any point, leading us to:

\noindent\doublebox{\begin{minipage}[t]{1\columnwidth - 2\fboxsep - 7.5\fboxrule - 1pt}%
\begin{align*}
L\in\QQ\; & \Rightarrow\quad\text{There is }C>0\text{ such that }\left|q_{n}L-p_{n}\right|\geq C\text{ for any }\frac{p_{n}}{q_{n}}\neq L\\
L\notin\QQ\; & \Rightarrow\quad\text{There are }\frac{p_{n}}{q_{n}}\in\QQ\text{ distinct such that }\left|q_{n}L-p_{n}\right|\to0.
\end{align*}
\end{minipage}}

In general, when looking for rational approximations with $p_{n},q_{n}$
as above, they are not going to be coprime. Letting $\tilde{q}_{n}=\frac{q_{n}}{\gcd\left(p_{n},q_{n}\right)},\;\tilde{p}_{n}=\frac{p_{n}}{\gcd\left(p_{n},q_{n}\right)}$
, we see that 
\[
\gcd\left(q_{n},p_{n}\right)\left|\tilde{q}_{n}L-\tilde{p}_{n}\right|=\left|q_{n}L-p_{n}\right|,
\]
so it is enough that $\left|q_{n}L-p_{n}\right|=o\left(\gcd\left(p_{n},q_{n}\right)\right)$
in order to get the condition $\left|\tilde{q}_{n}L-\tilde{p}_{n}\right|\to0$
above to hold. Since $\frac{\tilde{p}_{n}}{\tilde{q}_{n}}=\frac{p_{n}}{q_{n}}$,
we obtain the following irrationality test:
\begin{thm}
\label{thm:improved-rationality-test}Suppose that $p_{n},q_{n}\in\ZZ$
where $\frac{p_{n}}{q_{n}}$ is not eventually constant. If $\left|q_{n}L-p_{n}\right|=o\left(\gcd\left(p_{n},q_{n}\right)\right)$,
then $L$ is irrational.
\end{thm}

The importance of rational approximations derived from polynomial
continued fractions, as we shall see in \claimref{upper-bound}, is
that it gives us an upper bound on the error on the one hand, and
an easy way to check that $\frac{p_{n}}{q_{n}}$ is not eventually
constant on the other hand.

\medskip{}

\subsection*{The $\zeta\left(3\right)$ case:}

As an example to the irrationality testing, and how to improve the
approximations, let us consider $\zeta\left(3\right)$ which Ap\'ery
has shown to be irrational. We start with its standard rational approximations
\[
\frac{p_{n}}{q_{n}}:=\sum_{1}^{n}\frac{1}{k^{3}}\overset{n\to\infty}{\longrightarrow}\zeta\left(3\right).
\]

Multiplying the denominators, we can write $q_{n}=\left(n!\right)^{3}$
and $p_{n}=\sum_{1}^{n}\left(\frac{n!}{k}\right)^{3}$, both of which
in $\ZZ$. Trying to apply the irrationality test from the previous
section, we get that 
\[
\left|\zeta\left(3\right)-\frac{p_{n}}{q_{n}}\right|=\sum_{n+1}^{\infty}\frac{1}{k^{3}}=\Theta\left(\frac{1}{n^{2}}\right).
\]
This is, of course, far from what we need to prove irrationality,
since $\frac{1}{n^{2}}$ is much larger than $\frac{1}{q_{n}}=\frac{1}{\left(n!\right)^{3}}$.
Indeed, with denominator $q_{n}=\left(n!\right)^{3}$ we can always
find $p_{n}'$ such that $\left|\zeta\left(3\right)-\frac{p_{n}'}{q_{n}}\right|\leq\frac{1}{\left(n!\right)^{3}}\ll\frac{1}{n^{2}}$.
Hence, while the rational approximations choice above is easy to use
in general, it is not good enough to show irrationality.

One way to improve the approximation, as mentioned before, is by moving
to a reduced form of $\frac{p_{n}}{q_{n}}$. Taking instead the common
denominator in the sum above, we get that $\tilde{q}_{n}=\lcm\left[n\right]^{3}$,
where $\lcm\left[n\right]:=\lcm\left\{ 1,2,...,n\right\} $ and then
$\tilde{p}_{n}=\sum_{1}^{n}\left(\frac{\lcm\left[n\right]}{k}\right)^{3}$.
It is well known that $\ln\left(\lcm\left[n\right]\right)=n+o\left(n\right)$
(it follows from the prime number theorem, see \cite{apostol_introduction_1998}),
so that the new denominator $\tilde{q}_{n}=\lcm\left[n\right]^{3}\sim e^{3n}$
is much smaller than $\left(n!\right)^{3}$ and we basically get a
factorial reduction.

However, the error is still too big $\frac{1}{n^{2}}\gg\frac{1}{e^{3n}}$,
so even with this improvement, it is still not enough. One way to
improve this approximation even further, is to consider $\frac{p_{n}}{q_{n}}=\left(\sum_{1}^{n-1}\frac{1}{k^{3}}\right)+\frac{1}{2n^{2}}$.
While it doesn't change the common denominator too much, the error
becomes smaller, since
\[
\left|\zeta\left(3\right)-\frac{p_{n}}{q_{n}}\right|=\left|\left(\sum_{n}^{\infty}\frac{1}{k^{3}}\right)-\frac{1}{2n^{2}}\right|=\left|\left(\sum_{n}^{\infty}\frac{1}{k^{3}}\right)-\int_{n}^{\infty}\frac{1}{x^{3}}\dx\right|\leq\frac{1}{n^{3}}.
\]

While there is still a lot of room for improvement, there is reason
for optimism. The challenge lies in automating the process indefinitely
for all $r\in\NN$, aiming to reduce the error to $\frac{1}{n^{r}}$
for arbitrarily large $r$, with the hope of somewhere in the limit
getting the error to be small enough for our irrationality test.

The polynomial continued fractions become crucial at this juncture.
They not only furnish rational approximations but also offer iterative
means to refine them repeatedly, where in some cases leading to instances
where these improved approximations are good enough to prove the desired
irrationality. In particular, eventually in this paper, this polynomial
continued fractions will appear as ``horizontal rows'' of the conservative
matrix field. Going up in the rows will improve their convergence
rate, and finally, a diagonal argument will produce the good enough
approximation used to show that $\zeta\left(3\right)$ is irrational.

\newpage{}

\section{\label{sec:Generalized-continued-fractions}The polynomial continued
fraction}

We start with a generalization of the simple continued fractions,
which unsurprisingly, is called generalized continued fractions. These
can be defined over any topological field, though here we focus on
the complex field with its standard Euclidean metric, and more specifically
when the numerators and denominators are integers.
\begin{defn}[\textbf{(Generalized) continued fractions}]
Let $a_{n},b_{n}$ be a sequence of complex numbers. We will write
\[
a_{0}+\KK_{1}^{n}\frac{b_{i}}{a_{i}}:=a_{0}+\cfrac{b_{1}}{a_{1}+\cfrac{b_{2}}{a_{2}+\cfrac{b_{3}}{\ddots+\cfrac{b_{n}}{a_{n}+0}}}}\in\CC\cup\left\{ \infty\right\} ,
\]
and if the limit as $n\to\infty$ exists, we also will write
\[
a_{0}+\KK_{1}^{\infty}\frac{b_{i}}{a_{i}}:=\limfi n\left(a_{0}+\KK_{1}^{n}\frac{b_{i}}{a_{i}}\right).
\]

We call this type of expansion (both finite and infinite) a \textbf{continued
fraction expansion. }In case that $a_{0}+\KK_{1}^{\infty}\frac{b_{i}}{a_{i}}$
exists, we call the finite part $a_{0}+\KK_{1}^{n-1}\frac{b_{i}}{a_{i}}$
the \textbf{$n$-th convergents} for that expansion.

A continued fraction is called \textbf{simple continued fraction},
if $b_{i}=1$ for all $i$, $a_{0}\in\ZZ$ and $1\leq a_{i}\in\ZZ$
are positive integers for $i\geq1$.

A continued fraction is called \textbf{polynomial continued fraction},
if $b_{i}=b\left(i\right),\;a_{i}=a\left(i\right)$ for some polynomials
$a\left(x\right),b\left(x\right)\in\CC\left[x\right]$ and all $i\geq1$.

\medskip{}
\end{defn}

Simple continued fraction expansion is one of the main and basic tools
used in number theory when studying rational approximations, and their
convergents satisfy the inequality in Dirichlet's theorem that we
mentioned before (for more details, see chapter 3 in \cite{einsiedler_ergodic_2013}).
The coefficients $a_{i}$ in that expansion can be found using a generalized
Euclidean division algorithm, and it is well known that a number is
rational if and only if its simple continued fraction expansion is
finite. However, while we have an algorithm to find the (almost) unique
expansion, in general they can be very complicated without any known
patterns, even for ``nice'' numbers, for example:

\[
\pi=[3;7,15,1,292,1,1,1,2,1,3,1,14,2,1,1,2,2,2,2,1,84,...]=3+\cfrac{1}{7+\cfrac{1}{15+\cfrac{1}{1+\ddots}}}.
\]

When moving to generalized continued fractions, even when we assume
that both $a_{i}$ and $b_{i}$ are integers, we lose the uniqueness
property, and the rational if and only if finite property. What we
gain in return are more presentations for each number, where some
of them can be much simpler to use. For example, $\pi$ can be written
as 
\[
\pi=3+\KK_{1}^{\infty}\frac{\left(2n-1\right)^{2}}{6}.
\]
We want to study these presentations, and (hopefully) use them to
show interesting properties, e.g. prove irrationality for certain
numbers.

One such simple and interesting family of continued fractions is the
\textbf{Euler continued fractions} which are based on Euler's conversion
formula from infinite sums to continued fraction. This family was
introduced in \cite{david_euler_2023}, and in many cases our conservative
matrix fields in a sense will combine a parameterized family of such
Euler continued fractions. 
\begin{thm}[\textbf{Euler continued fractions}]
\label{thm:recurrence-roots}Let $h_{1},h_{2},f,a,b:\CC\to\CC$ be
functions such that 
\begin{align*}
b\left(x\right) & =-h_{1}\left(x\right)h_{2}\left(x\right)\\
f\left(x\right)a\left(x\right) & =f\left(x-1\right)h_{1}\left(x\right)+f\left(x+1\right)h_{2}\left(x+1\right).
\end{align*}
Then 
\[
\KK_{1}^{n}\frac{b\left(i\right)}{a\left(i\right)}=\frac{f\left(1\right)h_{2}\left(1\right)}{f\left(0\right)}\left(\frac{1}{\sum_{k=0}^{n}\frac{f\left(0\right)f\left(1\right)}{f\left(k\right)f\left(k+1\right)}\prod_{i=1}^{k}\left(\frac{h_{1}\left(i\right)}{h_{2}\left(i+1\right)}\right)}-1\right).
\]
\end{thm}

\begin{example}
\label{exa:Euler_zeta_3_example}The main example for these Euler continued fractions that we will
use, is to express values of the Riemann zeta function. Let $h_{1}\left(x\right)=h_{2}\left(x\right)=x^{d}$
for some $d\geq2$ and take $f\equiv1$. Then $\prod_{i=1}^{k}\left(\frac{h_{1}\left(i\right)}{h_{2}\left(i+1\right)}\right)=\prod_{i=1}^{k}\frac{i^{d}}{\left(i+1\right)^{d}}=\frac{1}{\left(k+1\right)^{d}}$,
so that
\[
\KK_{1}^{\infty}\frac{b\left(i\right)}{a\left(i\right)}=\frac{1}{\sum_{k=0}^{\infty}\frac{1}{\left(k+1\right)^{d}}}-1=\frac{1}{\zeta\left(d\right)}-1.
\]
Note also that since $a\left(0\right)=0^{d}+1^{d}=1$, we get that
\[
\cfrac{1}{a\left(0\right)+\cfrac{b\left(1\right)}{a\left(1\right)+\cfrac{b\left(2\right)}{a\left(2\right)+\ddots}}}=\frac{1}{1+\KK_{1}^{\infty}\frac{b\left(i\right)}{a\left(i\right)}}=\zeta\left(d\right).
\]
\end{example}

\bigskip{}

Moving on, one of the main tools used to study continued fractions
are \textbf{Mobius transformations}. Recall that given a $2\times2$
invertible matrix $M=\left(\begin{smallmatrix}a & b\\
c & d
\end{smallmatrix}\right)\in\GL_{2}\left(\CC\right)$ and a $z\in\CC$, the Mobius action is defined by
\[
M\left(z\right)=\frac{az+b}{cz+d}.
\]

In other words, we apply the standard matrix multiplication $\left(\begin{smallmatrix}a & b\\
c & d
\end{smallmatrix}\right)\left(\begin{smallmatrix}z\\
1
\end{smallmatrix}\right)=\left(\begin{smallmatrix}az+b\\
cz+d
\end{smallmatrix}\right)$ and project it onto $P^{1}\CC$ by dividing the $x$-coordinate by
the $y$-coordinate.

By this definition, it is easy to see that 
\[
\KK_{1}^{n}\frac{b_{i}}{a_{i}}=\frac{b_{1}}{a_{1}+\frac{b_{2}}{a_{2}+\frac{\ddots}{\frac{b_{n}}{a_{n}+0}}}}=\left(\begin{smallmatrix}0 & b_{1}\\
1 & a_{1}
\end{smallmatrix}\right)\left(\begin{smallmatrix}0 & b_{2}\\
1 & a_{2}
\end{smallmatrix}\right)\cdots\left(\begin{smallmatrix}0 & b_{n}\\
1 & a_{n}
\end{smallmatrix}\right)\left(0\right)
\]

This Mobius presentation allows us to show an interesting recurrence
relation on the numerators and denominators of the convergents, which
generalizes the well known recurrence on simple continued fractions.
\begin{lem}
\label{lem:gcf-recursion}Let $a_{n},b_{n}$ be a sequence of integers.
Define $M_{n}=\left(\begin{smallmatrix}0 & b_{n}\\
1 & a_{n}
\end{smallmatrix}\right)$ and set $\left(\begin{smallmatrix}p_{n}\\
q_{n}
\end{smallmatrix}\right)=\left(\prod_{1}^{n-1}M_{i}\right)\left(\begin{smallmatrix}0\\
1
\end{smallmatrix}\right)$. Then $\frac{p_{n}}{q_{n}}=\prod_{1}^{n-1}M_{i}\left(0\right)=\KK_{1}^{n-1}\frac{b_{i}}{a_{i}}$
are the convergents of the generalized continued fraction presentation.
More over, we have that $\left(\begin{smallmatrix}p_{n-1} & p_{n}\\
q_{n-1} & q_{n}
\end{smallmatrix}\right)=\prod_{1}^{n-1}M_{i}$, implying the same recurrence relation on $p_{n}$ and $q_{n}$ given
by 
\begin{align*}
p_{n+1} & =a_{n}p_{n}+b_{n}p_{n-1}\\
q_{n+1} & =a_{n}q_{n}+b_{n}q_{n-1},
\end{align*}
with starting condition $p_{0}=1,\;p_{1}=0$ and $q_{0}=0,\;q_{1}=1$.
\end{lem}

\begin{proof}
Left as an exercise.
\end{proof}
Now that we have this basic result, we can apply it to our irrationality
test as follows:
\begin{claim}
\label{claim:upper-bound}Let $a_{n},b_{n},p_{n},q_{n}$ be sequence
of integers satisfying the recurrence from \lemref{gcf-recursion}. 
\begin{enumerate}
\item If $\frac{p_{n}}{q_{n}}\to L$, then 
\[
\left|L-\frac{p_{n}}{q_{n}}\right|=\left|\sum_{k=n}^{\infty}\frac{\prod_{1}^{k}b_{i}}{q_{k}q_{k+1}}\right|\leq\sum_{k=n}^{\infty}\frac{\prod_{1}^{k}\left|b_{i}\right|}{\left|q_{k}q_{k+1}\right|}.
\]
\item Suppose in addition that the $b_{n}$ are nonzero. Then $\frac{p_{n}}{q_{n}}$
is not eventually constant, so that $\left|q_{n}L-p_{n}\right|=o\left(\gcd\left(p_{n},q_{n}\right)\right)$
implies that $L$ is irrational.
\end{enumerate}
\end{claim}

\begin{proof}
\begin{enumerate}
\item Let $M_{n}=\begin{pmatrix}0 & b_{n}\\
1 & a_{n}
\end{pmatrix}$ so that $\prod_{1}^{n-1}M_{i}=\left(\begin{smallmatrix}p_{n-1} & p_{n}\\
q_{n-1} & q_{n}
\end{smallmatrix}\right).$ For all $m\geq n$ we have
\[
L-\frac{p_{n}}{q_{n}}=L-\frac{p_{m+1}}{q_{m+1}}+\sum_{k=n}^{m}\left(\frac{p_{k+1}}{q_{k+1}}-\frac{p_{k}}{q_{k}}\right)=L-\frac{p_{m+1}}{q_{m+1}}-\sum_{k=n}^{m}\frac{\det\left(\begin{smallmatrix}p_{k} & p_{k+1}\\
q_{k} & q_{k+1}
\end{smallmatrix}\right)}{q_{k}q_{k+1}}.
\]
Under the assumption that $L-\frac{p_{m+1}}{q_{m+1}}\to0$ as $m\to\infty$,
we conclude that 
\[
\left|L-\frac{p_{n}}{q_{n}}\right|\leq\sum_{k=n}^{\infty}\left|\frac{\det\left(\begin{smallmatrix}p_{k} & p_{k+1}\\
q_{k} & q_{k+1}
\end{smallmatrix}\right)}{q_{k}q_{k+1}}\right|=\sum_{k=n}^{\infty}\frac{\prod_{1}^{k}\left|\det\left(M_{i}\right)\right|}{\left|q_{k}q_{k+1}\right|}=\sum_{k=n}^{\infty}\frac{\prod_{1}^{k}\left|b_{i}\right|}{\left|q_{k}q_{k+1}\right|}.
\]
\item In case that the $b_{k}\neq0$ for all $k$, then
\[
\left|\frac{p_{n+1}}{q_{n+1}}-\frac{p_{n}}{q_{n}}\right|=\left|\frac{\det\left(\begin{smallmatrix}p_{n} & p_{n+1}\\
q_{n} & q_{n+1}
\end{smallmatrix}\right)}{q_{n}q_{n+1}}\right|=\frac{\prod_{1}^{n}\left|b_{k}\right|}{\left|q_{n}q_{n+1}\right|}\neq0.
\]
It follows that the $\frac{p_{n}}{q_{n}}$ is not eventually constant,
so we can apply \thmref{improved-rationality-test} to prove this
claim.
\end{enumerate}
\end{proof}
The recursion relation of the $q_{i}$ suggests that the larger the
$\left|b_{i}\right|$ are, the faster the growth of $\left|q_{i}\right|$
is, and in general we expect it to be fast enough so that $\sum_{k=1}^{\infty}\frac{\prod_{1}^{k+1}\left|b_{i}\right|}{\left|q_{k}q_{k+1}\right|}$
will converge. However, it might still not be in $\left|q_{n}L-p_{n}\right|=o\left(\frac{1}{\left|q_{n}\right|}\right)$.
Hopefully, if the $gcd\left(p_{n},q_{n}\right)$ is large enough,
then it is in $o\left(\frac{\gcd\left(p_{n},q_{n}\right)}{\left|q_{n}\right|}\right)$,
which is enough to prove irrationality.

\bigskip{}

\newpage{}

\section{\label{sec:Definition-properties-CMF}The conservative matrix field
- definition and properties}

Until now we mainly looked at a single continued fractions $\KK_{1}^{\infty}\frac{b_{i}}{a_{i}}$,
and in particular where $a_{i}=a\left(i\right),b_{i}=b\left(i\right)$
with $a,b\in\ZZ\left[x\right]$. In this section we define the\textbf{
conservative matrix field}, which is a collection of such continued
fractions with interesting connections between them. 
\begin{defn}
A pair of matrices $M_{X}\left(x,y\right),M_{Y}\left(x,y\right)$
is called a \textbf{conservative matrix field} (or just \textbf{matrix
field} for simplicity), if
\begin{enumerate}
\item The entries of $M_{X}\left(x,y\right),M_{Y}\left(x,y\right)$ are
polynomial in $x,y$ ,
\item The matrices satisfy the conservativeness relation
\[
M_{X}\left(x,y\right)M_{Y}\left(x+1,y\right)=M_{Y}\left(x,y\right)M_{X}\left(x,y+1\right)\;\forall x,y,
\]
or in commutative diagram form:
\[
\xyR{.5pc}\xyC{0.5pc}\xymatrix{\left(x,y+1\right)\ar[rr]^{M_{X}\left(x,y+1\right)} &  & \left(x+1,y+1\right)\\
 & {\Huge\leftturn}\\
\left(x,y\right)\ar[rr]_{M_{X}\left(x,y\right)}\ar[uu]^{M_{Y}\left(x,y\right)} &  & \left(x+1,y\right)\ar[uu]_{M_{Y}\left(x+1,y\right)}
}
\xyR{1pc}\xyC{1pc}
\]
\end{enumerate}
\end{defn}

\begin{rem}
In the definition of conservative matrix field we only use 2 matrices
for simplicity, though a similar definition can be given for $n$
polynomial matrices. Additionally, while we don't add it as a condition
to the matrix field, in many of the interesting examples that we have
found so far, at least one of the direction $M_{X}$ or $M_{Y}$ is
``almost'' in continued fraction form, namely $\begin{pmatrix}0 & b\left(x,y\right)\\
1 & a\left(x,y\right)
\end{pmatrix}$.
\end{rem}

The name \emph{conservative matrix field} arose from the resemblance
to \emph{conservative vector field}. Visualizing the commutative diagram's
corners as points in the plane, the notion is that traveling along
the bottom and then right edge or the left and then top edge yields
the same product, which essentially is the behaviour of standard conservative
vector fields. Retaining this intuitive connection, led to the adoption
of the name conservative matrix field. This similarity can be formalized
using cohomology language - both are 1-cocycles with the appropriate
groups, though we will not use this cohomology theory in this paper,
and keep it elementary.

One of the main differences, is that moving left or down along the
matrix field amounts to multiplying by $M_{X}\left(x,y\right)^{-1},M_{Y}\left(x,y\right)^{-1}$
respectively, which aren't necessarily invertible. However, they will
be invertible in most cases (e.g. for $x,y\geq1$), and with that
in mind, we define:
\begin{defn}[Potential matrix]
\label{def:potential_matrix}Given a matrix field $M_{X},M_{Y}$,
and initial position $\left(n_{0},m_{0}\right)\in\ZZ^{2}$, we define
the potential $S\left(n,m\right)=S_{n_{0},m_{0}}\left(n,m\right)$
matrix for $n\geq n_{0},m\geq m_{0}$ by

\[
S\left(n,m\right)=\prod_{n_{0}}^{n-1}M_{X}\left(k,m_{0}\right)\cdot\prod_{m_{0}}^{m-1}M_{Y}\left(n,k\right).
\]
Note that the potential is independent of the choice of path from
$\left(n_{0},m_{0}\right)$ to $\left(n,m\right)$.
\end{defn}

\medskip{}

Given a path to infinity $\left(n_{i},m_{i}\right)$, it is natural
to ask whether $\lim S\left(n_{i},m_{i}\right)\left(0\right)$ converges,
and how does changing the path affects the limit. In particular, if
$\alpha=\lim S\left(n_{i},m_{i}\right)\left(0\right)$ along some
path, can we extract other properties of $\alpha$ from the rest of
the matrix field? For example, in \secref{The-z3-case} we will construct
such a matrix field for $\alpha=\zeta\left(3\right)$, starting from
its Euler continued fraction (from \exaref{Euler_zeta_3_example})
on the $Y=1$ line, then see that the limits on the $Y=m$ lines converge
to $\sum_{m}^{\infty}\frac{1}{k^{3}}$, and the diagonal line $X=Y$
can be used to define another continued fraction presentation which
converges to $\zeta\left(3\right)$ fast enough to prove its irrationality.

With this intuition in mind, we start with a construction for a specific
family of matrix fields with many interesting properties in \subsecref{matrix-field-construction},
where in particular each row is a polynomial continued fraction. We
then ``twist'' it a little bit (which is formally a ``coboundary
equivalence'' from cohomology), to get a matrix field which is easier
to work with. Then in \subsecref{The-dual-matrix-field} we find out
how every such matrix field comes with its dual, which is in a sense
a reflection through the $x=y$ line. Once we have this dual matrix
field, we study the numerators and denominators of the continued fractions
in that matrix field, and in particular find their greatest common
divisors. Finally we show how to put everything together in \secref{The-z3-case}
to show that $\zeta\left(3\right)$ is irrational.
\begin{rem}
Before continuing to this interesting construction, we note first
that there are several ``trivial `` constructions.

First, if we remove the polynomiality condition, then starting with
any potential function $S\left(n,m\right)$ of invertible matrices
for say $n,m\geq0$, we may define 
\begin{align*}
M_{X}\left(n,m\right) & =S\left(n,m\right)^{-1}S\left(n+1,m\right)\\
M_{Y}\left(n,m\right) & =S\left(n,m\right)^{-1}S\left(n,m+1\right),
\end{align*}
to automatically get the conservativeness condition above true.

However, even with the polynomiality condition, there are still trivial
constructions. 
\begin{enumerate}
\item \textbf{\uline{The commutative construction}}: If $M_{X}\left(x,y\right)=A_{X}\left(x\right)$
and $M_{Y}\left(x,y\right)=A_{Y}\left(y\right)$ only depend on $x$
and $y$ respectively, then the conservative condition will become
the commutativity condition
\[
A_{X}\left(x\right)A_{Y}\left(y\right)=M_{X}\left(x,y\right)M_{Y}\left(x+1,y\right)=M_{Y}\left(x,y\right)M_{X}\left(x,y+1\right)=A_{Y}\left(y\right)A_{X}\left(x\right).
\]
One way to construct such examples, is to take two polynomials in
commuting variables $f_{X}\left(u,v\right),\;f_{Y}\left(u,v\right)$
, and a given $2\times2$ matrix $B$, and simply define
\begin{align*}
A_{X}\left(x\right) & =f_{X}\left(xI,B\right)\\
A_{Y}\left(y\right) & =f_{Y}\left(yI,B\right).
\end{align*}
\item \textbf{\uline{The 1-dimension inflation}}: If $M\left(t\right)$
is any polynomial matrix in a single variable, then
\begin{align*}
M_{X}\left(x,y\right) & =M_{Y}\left(x,y\right)=M\left(x+y\right)
\end{align*}
is a conservative matrix field.
\end{enumerate}
There are other constructions as well, some with more interesting
properties, and some more ``trivial''. It is still not clear what
is the condition, and if there is such, that makes a conservative
matrix field into an ``interesting'' one. As for now we keep the
current algebraic definition, but as mentioned before, there should
probably be an analytic comopnent as well to the definition, which
relates to the limit of the potential on different pathes to infinity
(for example, see \thmref{limit_on_each_line} where we go to infinity
along a fixed row $y=m$).
\end{rem}

\medskip{}

\begin{rem}
In our investigations of these matrix fields, in many cases we came
across a ``natural'' family of polynomial matrices $M_{X}\left(x,y\right)$
in $x$ (and in particular of continued fractions), and asked whether
there exists a polynomial matrix $M_{Y}\left(x,y\right)$ which completes
it to a matrix field. For that purpose we created a python code which
looks for such solutions, which you can find in \cite{ramanujan_machine_research_group_ramanujan_2023}.
In this paper, we focus on matrix field for the construction given
in the next section, however, we have found many other with interesting
properties which do not seem to fall under that construction.
\end{rem}

\newpage{}

\subsection{\label{subsec:matrix-field-construction}A matrix field construction}

We begin with an interesting construction of a family of conservative
matrix fields.
\begin{defn}
\label{def:conjugate}We say that two polynomial $f,\bar{f}\in\CC\left[x,y\right]$
are \textbf{conjugate}, if they satisfy:
\begin{enumerate}
\item \textbf{\uline{Linear condition}}: 
\[
f\left(x+1,y-1\right)-\bar{f}\left(x,y-1\right)=f\left(x,y\right)-\bar{f}\left(x+1,y\right).
\]
When this condition holds, we denote the expression above by $a_{f,\bar{f}}\left(x,y\right):=a\left(x,y\right)$.
\item \textbf{\uline{Quadratic condition}}: 
\[
\left(f\bar{f}\right)\left(x,y\right)-\left(f\bar{f}\right)\left(0,y\right)=\left(f\bar{f}\right)\left(x,0\right)-\left(f\bar{f}\right)\left(0,0\right).
\]
In other words, there are no mixed monomials $x^{n}y^{m}$ where $n,m>0$
in $\left(f\bar{f}\right)\left(x,y\right)$, so we can write it as
$\fbf\left(x,y\right)=b_{X}\left(x\right)+b_{Y}\left(y\right)$.
\end{enumerate}
Given two such conjugate polynomials, and a decomposition $\fbf\left(x,y\right)=b_{X}\left(x\right)+b_{Y}\left(y\right)$,
setting $b\left(x\right)=b_{X}\left(x\right)$ we write
\begin{align*}
M_{X}^{cf}\left(x,y\right) & =\left(\begin{smallmatrix}0 & b\left(x\right)\\
1 & a\left(x,y\right)
\end{smallmatrix}\right)\\
M_{Y}^{cf}\left(x,y\right) & =\left(\begin{smallmatrix}\bar{f}\left(x,y\right) & b\left(x\right)\\
1 & f\left(x,y\right)
\end{smallmatrix}\right).
\end{align*}
The $cf$ indicates that $M_{X}^{cf}$ has continued fraction form.
We will shortly change it a little bit and remove these $cf$.
\end{defn}

\medskip{}

\begin{rem}
\label{rem:ff(0)=00003D0}In the case where $b_{X}\left(x\right)=\fbf\left(x,0\right)$,
the $y=1$ line is a continued fraction with $b_{i}=\left(f\bar{f}\right)\left(i,0\right)$
and $a_{i}=f\left(i+1,0\right)-\bar{f}\left(i,0\right)$. This is
in the trivial Euler family defined in \thmref{recurrence-roots},
where $h_{1}\left(x\right)=f\left(x,0\right)$ and $h_{2}\left(x\right)=-\bar{f}\left(x,0\right)$.
This will come into play later (see part \enuref{first-line-polynomials}
in \claimref{dual-field-identities}). 
\end{rem}

\begin{example}[\textbf{The $\zeta\left(3\right)$ matrix field}]
\label{exa:zeta-3-matrix-field}The main example that we should have in mind is a matrix field for
$\zeta\left(3\right)$ defined by
\begin{align*}
f\left(x,y\right) & =x^{3}+2x^{2}y+2xy^{2}+y^{3}=\frac{y^{3}-x^{3}}{y-x}\left(y+x\right)\\
\bar{f}\left(x,y\right) & =-x^{3}+2x^{2}y-2xy^{2}+y^{3}=\frac{y^{3}+x^{3}}{y+x}\left(y-x\right)=f\left(-x,y\right)\\
\left(f\bar{f}\right)\left(x,y\right) & =y^{6}-x^{6}\\
b\left(x\right) & =\left(f\bar{f}\right)\left(x,0\right)=-x^{6}\\
a\left(x,y\right) & =x^{3}+\left(1+x\right)^{3}+2y\left(y-1\right)\left(2x+1\right).
\end{align*}
In particular, as in the remark above, the $y=1$ line is the Euler
continued fraction with $b\left(n\right)=-n^{6}$ and $a\left(n,0\right)=n^{3}+\left(1+n\right)^{3}$,
which we already saw in \exaref{Euler_zeta_3_example} that its convergents
are $\left(\sum_{1}^{n}\frac{1}{k^{3}}\right)^{-1}-1$. We will see
in \secref{The-z3-case}, and more specifically in \corref{zeta_3_limit_on_lines},
that for any fixed integer $y=m\geq1$, the continued fraction with
$b_{n}=b\left(n\right)$ and $a_{n}=a\left(n,m\right)$ converges
to $\frac{1}{\sum_{m}^{\infty}\frac{1}{k^{3}}}-1$.

The polynomial matrices in this matrix field are 
\begin{align*}
M_{X}^{cf}\left(x,y\right) & =\left(\begin{smallmatrix}0 & b\left(x\right)\\
1 & a\left(x,y\right)
\end{smallmatrix}\right)=\left(\begin{smallmatrix}0 & -x^{6}\\
1 & x^{3}+\left(1+x\right)^{3}+2y\left(y-1\right)\left(2x+1\right)
\end{smallmatrix}\right)\\
M_{Y}^{cf}\left(x,y\right) & =\left(\begin{smallmatrix}\bar{f}\left(x,y\right) & b\left(x\right)\\
1 & f\left(x,y\right)
\end{smallmatrix}\right)=\left(\begin{smallmatrix}\frac{y^{3}+x^{3}}{y+x}\left(y-x\right) & -x^{6}\\
1 & \frac{y^{3}-x^{3}}{y-x}\left(y+x\right)
\end{smallmatrix}\right)
\end{align*}
and the first few of them are 
\begin{align*}
\xyR{1.5pc}\xyC{1.5pc}\xymatrix{\left(*\right) &  &  & \left(*\right) &  &  & \left(*\right) &  & \left(*\right)\\
\\
*+[F]{\left(1,3\right)}\ar[rrr]_{\left(\begin{smallmatrix}0 & -1\\
1 & 45
\end{smallmatrix}\right)}\ar[uu]^{\left(\begin{smallmatrix}14 & -1\\
1 & 52
\end{smallmatrix}\right)} &  &  & *+[F]{\left(2,3\right)}\ar[rrr]_{\left(\begin{smallmatrix}0 & -2^{3}\\
1 & 95
\end{smallmatrix}\right)}\ar[uu]^{\left(\begin{smallmatrix}7 & -1\\
1 & 95
\end{smallmatrix}\right)} &  &  & *+[F]{\left(3,3\right)}\ar[rr]_{\left(\begin{smallmatrix}0 & -3^{3}\\
1 & 175
\end{smallmatrix}\right)}\ar[uu]^{\left(\begin{smallmatrix}0 & -1\\
1 & 168
\end{smallmatrix}\right)} &  & \left(*\right)\\
\\
*+[F]{\left(1,2\right)}\ar[rrr]_{\left(\begin{smallmatrix}0 & -1\\
1 & 21
\end{smallmatrix}\right)}\ar[uu]^{\left(\begin{smallmatrix}3 & -1\\
1 & 21
\end{smallmatrix}\right)} &  &  & *+[F]{\left(2,2\right)}\ar[rrr]_{\left(\begin{smallmatrix}0 & -2^{3}\\
1 & 55
\end{smallmatrix}\right)}\ar[uu]^{\left(\begin{smallmatrix}0 & -1\\
1 & 48
\end{smallmatrix}\right)} &  &  & *+[F]{\left(3,2\right)}\ar[rr]_{\left(\begin{smallmatrix}0 & -3^{3}\\
1 & 119
\end{smallmatrix}\right)}\ar[uu]^{\left(\begin{smallmatrix}-7 & -1\\
1 & 95
\end{smallmatrix}\right)} &  & \left(*\right)\\
\\
*+[F]{\left(1,1\right)}\ar[rrr]_{\left(\begin{smallmatrix}0 & -1\\
1 & 9
\end{smallmatrix}\right)}\ar[uu]^{\left(\begin{smallmatrix}0 & -1\\
1 & 6
\end{smallmatrix}\right)} &  &  & *+[F]{\left(2,1\right)}\ar[rrr]_{\left(\begin{smallmatrix}0 & -2^{3}\\
1 & 35
\end{smallmatrix}\right)}\ar[uu]^{\left(\begin{smallmatrix}-3 & -1\\
1 & 21
\end{smallmatrix}\right)} &  &  & *+[F]{\left(3,1\right)}\ar[rr]_{\left(\begin{smallmatrix}0 & -3^{3}\\
1 & 91
\end{smallmatrix}\right)}\ar[uu]^{\left(\begin{smallmatrix}-14 & -1\\
1 & 52
\end{smallmatrix}\right)} &  & \left(*\right)
}
\end{align*}
\end{example}

We continue to show that this general construction produces conservative
matrix fields.
\begin{thm}
\label{thm:matrix-field-structure}Given polynomials $f,\bar{f},a,b$
where $b\not\equiv0$, define the matrices
\begin{align*}
M_{X}^{cf}\left(x,y\right) & =\left(\begin{smallmatrix}0 & b\left(x\right)\\
1 & a\left(x,y\right)
\end{smallmatrix}\right)\\
M_{Y}^{cf}\left(x,y\right) & =\left(\begin{smallmatrix}\bar{f}\left(x,y\right) & b\left(x\right)\\
1 & f\left(x,y\right)
\end{smallmatrix}\right).
\end{align*}
The following hold:
\begin{enumerate}
\item The polynomials $f,\bar{f},a,b$ are as in \defref{conjugate} if
and only if the conservativeness condition holds
\[
M_{X}^{cf}\left(x,y\right)M_{Y}^{cf}\left(x+1,y\right)=M_{Y}^{cf}\left(x,y\right)M_{X}^{cf}\left(x,y+1\right)\;\forall x,y.
\]
\item The determinants of $M_{X}^{cf}\left(x,y\right),M_{Y}^{cf}\left(x,y\right)$
are only functions of $x,y$ respectively, and more specifically:
\begin{align*}
\det\left(M_{X}^{cf}\left(x,y\right)\right) & =-b\left(x\right)=-b_{X}\left(x\right)\\
\det\left(M_{Y}^{cf}\left(x,y\right)\right) & =b_{Y}\left(y\right).
\end{align*}
\end{enumerate}
\end{thm}

\begin{proof}
\begin{enumerate}
\item Considering the conservative condition, we obtain 
\begin{align*}
M_{X}^{cf}\left(x,y\right)M_{Y}^{cf}\left(x+1,y\right) & =\left(\begin{smallmatrix}0 & b\left(x\right)\\
1 & a\left(x,y\right)
\end{smallmatrix}\right)\left(\begin{smallmatrix}\bar{f}\left(x+1,y\right) & b\left(x+1\right)\\
1 & f\left(x+1,y\right)
\end{smallmatrix}\right)=\left(\begin{smallmatrix}b\left(x\right) & b\left(x\right)f\left(x+1,y\right)\\
\bar{f}\left(x+1,y\right)+a\left(x,y\right) & b\left(x+1\right)+a\left(x,y\right)f\left(x+1,y\right)
\end{smallmatrix}\right)\\
M_{Y}^{cf}\left(x,y\right)M_{X}^{cf}\left(x,y+1\right) & =\left(\begin{smallmatrix}\bar{f}\left(x,y\right) & b\left(x\right)\\
1 & f\left(x,y\right)
\end{smallmatrix}\right)\left(\begin{smallmatrix}0 & b\left(x\right)\\
1 & a\left(x,y+1\right)
\end{smallmatrix}\right)=\left(\begin{smallmatrix}b\left(x\right) & b\left(x\right)\left(\bar{f}\left(x,y\right)+a\left(x,y+1\right)\right)\\
f\left(x,y\right) & b\left(x\right)+f\left(x,y\right)a\left(x,y+1\right)
\end{smallmatrix}\right).
\end{align*}
The equality at the bottom left and top right corners are equivalent
to 
\begin{align*}
a\left(x,y\right) & =f\left(x,y\right)-\bar{f}\left(x+1,y\right)\\
a\left(x,y+1\right)= & f\left(x+1,y\right)-\bar{f}\left(x,y\right),
\end{align*}
which is the linear condition. Given this condition, the bottom right
corner equality is equivalent to 
\[
b\left(x+1\right)-b\left(x\right)=f\left(x,y\right)a\left(x,y+1\right)-a\left(x,y\right)f\left(x+1,y\right)=\fbf\left(x+1,y\right)-\fbf\left(x,y\right).
\]
This means that $b\left(x\right)-\fbf\left(x,y\right)$ is independent
of $x$, or equivalently there is a decomposition $\fbf\left(x,y\right)=b_{X}\left(x\right)+b_{Y}\left(y\right)$
which is exactly the quadratic condition.
\item Simple computation.
\end{enumerate}
\end{proof}
\begin{rem}
Once we have a matrix field as in \thmref{matrix-field-structure},
changing the origin to $\left(\alpha,\beta\right)$ is equivalent
to looking at the polynomials 
\begin{align*}
f_{\alpha,\beta}\left(x,y\right) & :=f\left(x+\alpha,y+\beta\right),\\
\bar{f}_{\alpha,\beta}\left(x,y\right) & :=\bar{f}\left(x+\alpha,y+\beta\right).
\end{align*}
The corresponding decomposition $\fbf\left(x,y\right)=b_{X}\left(x\right)+b_{Y}\left(y\right)$
then becomes 
\[
\left(f_{\alpha,\beta}\bar{f}_{\alpha,\beta}\right)\left(x,y\right)=b_{X}\left(x+\alpha\right)+b_{Y}\left(y+\beta\right).
\]
We are mainly interested in moving along the integer points of the
matrix field, and if $\alpha,\beta$ are integers, then this translation
is done inside this integer lattice. In particular, up to this integer
translation we may assume that $b_{X}\left(x\right)\neq0$ for $x>0$
(though for now we allow $b_{X}\left(0\right)=0$). We similarly assume
that in the $y$-direction. Thus, all the matrices $M_{X}\left(x,y\right),M_{Y}\left(x,y\right)$
for $x,y>0$ are invertible.
\end{rem}

\begin{example}
\label{exa:matrix-field-examples}There are many examples for this
construction, and we give some of them below.

Starting with trivial solutions, whenever $f\left(x,y\right)=A\left(y\right)+B\left(x\right)$
and $\bar{f}\left(x,y\right)=A\left(y\right)-B\left(x\right)$, it
is easy to check that the linear and quadratic conditions hold. However,
in this case we have that $a\left(x,y\right):=f\left(x,y\right)-\bar{f}\left(x+1,y\right)=B\left(x\right)+B\left(x+1\right)$
doesn't depend on $y$, so all the horizontal lines in the matrix
fields are the same, and therefore in a sense it is degenerate. Fortunately,
there are many cases of nondegenerate matrix fields.\\

In the following examples, for each pair $f,\bar{f}$, we also add
the $b\left(x\right),a\left(x,y\right)$ appearing as the continued
fraction on the horizontal lines. In particular, as we saw in \remref{ff(0)=00003D0},
when $b\left(x\right)=\fbf\left(x,0\right)$, the $y=1$ line is in
the Euler Family from \thmref{recurrence-roots}, namely $b\left(n\right)=-h_{1}\left(n\right)\times h_{2}\left(n\right)$
and $a\left(n\right)=h_{1}\left(n\right)+h_{2}\left(n+1\right)$.
In these cases we can convert it to an infinite sum and hopefully
use it to compute the value of the continued fraction, which we add
in the examples below (up to a Mobius map). Further more, in many
cases we think of $\bar{f}$ as an image under some nice linear map
$g\mapsto\bar{g}$ of $f$, and when this is the case, we will give
this linear map instead of $\bar{f}$.
\begin{enumerate}
\item When both $f,\bar{f}$ are linear themselves, solving the linear and
quadratic conditions in \defref{conjugate} is elementary (which we
leave as an exercise). There is one nontrivial family
\begin{align*}
f\left(x,y\right) & =A\left(x+y\right)+C\\
\bar{f}\left(x,y\right) & =\bar{A}\left(x-y\right)+\bar{C}
\end{align*}
where $A,C,\bar{A},\bar{C}$ above are the parameters of the family.
\begin{enumerate}
\item Taking $f\left(x,y\right)=x+y$ and $\bar{f}\left(x,y\right)=x-y$,
we get $b\left(x\right)=x^{2}$ and $a\left(x,y\right)=2y-1$. In
$y=1$ we get the continued fraction
\[
\KK_{1}^{\infty}\frac{n^{2}}{1}=\KK_{1}^{\infty}\frac{-\left(-n\right)\times n}{\left(-n\right)+\left(n+1\right)}=\frac{1}{\sum_{k=0}^{\infty}\prod_{i=1}^{k}\frac{-i}{i+1}}-1=\frac{1}{\sum_{k=0}^{\infty}\frac{\left(-1\right)^{k}}{k+1}}-1=\frac{1-\ln\left(2\right)}{\ln\left(2\right)}.
\]
Taking $\bar{f}\left(x,y\right)=y-x$ instead (so it is a ``trivial''
solution), we get $b\left(x\right)=-x^{2}$ and $a\left(x,y\right)=2x+1$.
Since $a$ is independent of $y$, all the horizontal lines in the
matrix field are the same, so in a sense it is degenerate. Moreover,
trying to compute the continued fraction produces
\[
\KK_{1}^{\infty}\frac{-n^{2}}{2n+1}=\KK_{1}^{\infty}\frac{-n\times n}{n+\left(n+1\right)}=\frac{1}{\sum_{k=0}^{\infty}\prod_{i=1}^{k}\frac{i}{i+1}}-1=\frac{1}{\sum_{k=0}^{\infty}\frac{1}{k+1}}-1=-1,
\]
since the harmonic sum $\sum_{0}^{\infty}\frac{1}{k+1}$ diverges
to infinity.
\item For $f\left(x,y\right)=x+y$ and $\bar{f}\left(x,y\right)=1$ (which
we can think of as $\frac{\partial f}{\partial x}=\frac{\partial f}{\partial y}=\bar{f}$),
we get $b\left(x\right)=x,\;a\left(x,y\right)=x+y-1$, and in the
$y=1$ case we get 
\[
\KK_{1}^{\infty}\frac{n}{n}=\KK_{1}^{\infty}\frac{-\left(\left(-1\right)\times n\right)}{\left(-1\right)+\left(n+1\right)}=\frac{1}{\sum_{k=0}^{\infty}\prod_{i=1}^{k}\frac{-1}{i+1}}-1=\frac{1}{e-1}.
\]
\end{enumerate}
\item When $f,\bar{f}$ have degree at most 2, then we have the following
families of examples (as function of $C$):{\tiny{}
\[
\begin{array}{c|c|c|c|c}
\text{operation} & f\left(x,y\right) & a\left(x,y\right) & b\left(x\right) & \text{Euler family}\;\left(a\left(x,1\right)\right)\\
\hline \bar{g}\left(x,y\right)=-g\left(-x,y\right) & x^{2}+xy+\frac{y^{2}}{2}+C\left(x+y\right) & \left(x+1\right)^{2}+x^{2}+y\left(y-1\right)+C\left(2y-1\right) & -x^{2}\left(x^{2}-C^{2}\right) & \left(x+1\right)\left(x+1+C\right)+x\left(x-C\right)\\
\bar{g}\left(x,y\right)=g\left(-x,y\right) & x^{2}+2xy+2y^{2}+C\left(2y-x\right) & \left(2x+1\right)\left(2y-1+C\right) & x^{2}\left(x^{2}-C^{2}\right) & \left(x+1\right)\left(x+1+C\right)-x\left(x-C\right)\\
\bar{g}\left(x,y\right)=g\left(x,-y\right) & x^{2}+2xy+2y^{2}+C\left(x+y\right) & \left(2x+1+C\right)\left(2y-1\right) & x^{2}\left(x+C\right)^{2} & \left(x+1\right)\left(x+C+1\right)-x\left(x+C\right)\\
\bar{g}\left(x,y\right)=-g\left(x,-y\right) & \frac{2x^{2}+2xy+y^{2}+C\left(2x+y\right)}{2} & C\left(2x+1\right)+x^{2}+\left(x+1\right)^{2}+y\left(y-1\right) & -x^{2}\left(x+C\right)^{2} & \left(x+1\right)\left(x+C+1\right)+x\left(x+C\right)
\end{array}
\]
}In particular, when taking $C=0$, the $y=1$ line is either $b\left(x\right)=-x^{4}$
and $a\left(x\right)=x^{2}+\left(1+x\right)^{2}$, or $b\left(x\right)=x^{4}$
and $a\left(x\right)=\left(x+1\right)^{2}-x^{2}$. The continued fraction
will eventually be transformed (after the right Mobius action) to
the sums $\sum_{1}^{\infty}\frac{1}{n^{2}}$ and $\sum_{1}^{\infty}\frac{\left(-1\right)^{n}}{n^{2}}$
which are $\zeta\left(2\right)$ and $\frac{1}{2}\zeta\left(2\right)$
respectively.
\item In degree 3, with the action $\bar{g}\left(x,y\right)\mapsto g\left(-x,y\right)$,
we have the family 
\begin{align*}
f\left(x,y\right) & =x^{3}+2x^{2}\left(y-C\right)+2x\left(y-C\right)^{2}+\left(y-C\right)^{3}-\left(x+y-C\right)C^{2}\\
b\left(x\right) & =-x^{2}\left(x-C\right)^{2}\left(x+C\right)^{2}\\
a\left(x,y\right) & =x\left(x-C\right)^{2}+\left(x+1\right)\left(x+1+C\right)^{2}+\left(1+2x\right)\left(y-1-2C\right)2y.
\end{align*}
When $y=1$ we get a continued fraction in the Euler family with $h_{1}\left(x\right)=x\left(x-C\right)^{2}$
and $h_{2}\left(x\right)=x\left(x+C\right)^{2}$. In particular, in
the case where $C=0$ we simply get the matrix field for $\zeta\left(3\right)$
mentioned in \exaref{zeta-3-matrix-field}.
\end{enumerate}
\end{example}

\begin{rem}
Once we have a pair of conjugate polynomials $f,\bar{f}$, there are
several ways to generate more such pairs. Indeed, we already saw the
translation of parameters above, but another simple way is just to
take $cf,c\bar{f}$ for some $0\neq c\in\CC$. Another less trivial
way is to look at the pair $\left(f\left(y,x\right),\;-\bar{f}\left(y,x\right)\right)$.
We shall see in \subsecref{The-dual-matrix-field} how this new pair
is hidden in the same conservative matrix field.
\end{rem}

\medskip{}

\subsection*{Twisting the matrix field}

Right now, while the $M_{X}^{cf}$ matrix has the known continued
fraction form, the $M_{Y}^{cf}$ matrices have this new unknown form
$\left(\begin{smallmatrix}\bar{f}\left(x,y\right) & b\left(x\right)\\
1 & f\left(x,y\right)
\end{smallmatrix}\right)$. However, there are hidden continued fractions in $M_{Y}^{cf}$ as
well, and both $M_{X}^{cf},M_{Y}^{cf}$ are defined very similarly.
For that, we use the following notations.
\begin{notation}
We define:
\[
U_{\alpha}=\left(\begin{smallmatrix}1 & \alpha\\
0 & 1
\end{smallmatrix}\right)\qquad D_{\alpha}=\left(\begin{smallmatrix}\alpha & 0\\
0 & 1
\end{smallmatrix}\right)\qquad\tau=\left(\begin{smallmatrix}0 & 1\\
1 & 0
\end{smallmatrix}\right).
\]
For any matrix $M$, we will write the isomorphism $M\mapsto M^{\tau}=\tau M\tau^{-1}$
(and note that $\tau^{2}=Id$, so that $\tau^{-1}=\tau$). More specifically,
we have that $\left(\begin{smallmatrix}a & b\\
c & c
\end{smallmatrix}\right)^{\tau}=\left(\begin{smallmatrix}d & c\\
b & a
\end{smallmatrix}\right)$ is just switching the rows and switching the columns, and in particular
$U_{\alpha}^{\tau}=U_{\alpha}^{tr}$.
\end{notation}

\medskip{}

With these notations we get:
\begin{alignat*}{2}
M_{X}^{cf}\left(x,y\right) & =\left(\begin{smallmatrix}0 & b\left(x\right)\\
1 & f\left(x,y\right)-\bar{f}\left(x+1,y\right)
\end{smallmatrix}\right) &  & =D_{b_{X}\left(x\right)}\cdot\tau\cdot U_{f\left(x,y\right)}\cdot U_{-\bar{f}\left(x+1,y\right)}\\
M_{Y}^{cf}\left(x,y\right) & =\left(\begin{smallmatrix}\bar{f}\left(x,y\right) & b\left(x\right)\\
1 & f\left(x,y\right)
\end{smallmatrix}\right) &  & =U_{\bar{f}\left(x,y\right)}\cdot D_{-b_{Y}\left(y\right)}\cdot\tau\cdot U_{f\left(x,y\right)},
\end{alignat*}
so that $M_{X}^{cf}$ and $M_{Y}^{cf}$ are ``almost'' the same.
There is some ``cyclic permutation'' and after it they have a similar
structure, with related parameters. In particular the $M_{Y}^{cf}$
is also a continued fraction sequence in disguise.

Already this notation suggests at least two directions to study these
matrix fields. The first, is that there is sort of duality between
the $X$ and $Y$ directions, which we will study in \subsecref{The-dual-matrix-field}.

For the second, recall that we are interested in products like $\prod_{1}^{k}M_{X}\left(i,y\right)$,
though we would like to use only invertible matrices, namely when
$b_{X}\left(x\right)\neq0$. This means, that we might not use $M_{X}\left(0,y\right)=D_{b_{X}\left(0\right)}\cdot\left(\tau\cdot U_{f\left(0,y\right)}\cdot U_{-\bar{f}\left(1,y\right)}\right)$
just because the first factor $D_{b_{X}\left(k\right)}$ is not invertible,
and therefore ``lose'' the information coming from the second factor
\[
\tau\cdot U_{f\left(0,y\right)-\bar{f}\left(1,y\right)}=\begin{pmatrix}0 & 1\\
1 & a\left(0,y\right)
\end{pmatrix}.
\]
Moreover, adding this part at the beginning we get the natural extension
of the continued fraction
\[
\left[\begin{pmatrix}0 & 1\\
1 & a\left(0,y\right)
\end{pmatrix}\cdot\prod_{1}^{k}M_{X}\left(i,y\right)\right]\left(0\right)=\frac{1}{a\left(0,y\right)+\frac{b\left(1\right)}{a\left(1,y\right)+\cfrac{b\left(2\right)}{a\left(2,y\right)+\ddots}}}.
\]

With this in mind, we twist our matrix field to a new form, which
will make the computations later much simpler:
\begin{defn}
Let $f,\bar{f}$ be conjugate polynomials and $M_{X},M_{Y}$ as in
\defref{conjugate}. Define
\begin{align*}
M_{X}\left(x,y\right) & :=D_{b\left(x\right)}^{-1}M_{X}^{cf}\left(x,y\right)D_{b\left(x+1\right)}=\tau U_{f\left(x,y\right)}U_{-\bar{f}\left(x+1,y\right)}D_{b\left(x+1\right)}=\left(\begin{smallmatrix}0 & 1\\
b\left(x+1\right) & f\left(x,y\right)-\bar{f}\left(x+1,y\right)
\end{smallmatrix}\right)\\
M_{Y}\left(x,y\right) & :=D_{b\left(x\right)}^{-1}M_{Y}^{cf}\left(x,y\right)D_{b\left(x\right)}=U_{f\left(x,y\right)}^{\tau}\tau D_{-\left(f\bar{f}\right)\left(0,y\right)}U_{\bar{f}\left(x,y\right)}^{\tau}=\left(\begin{smallmatrix}\bar{f}\left(x,y\right) & 1\\
b\left(x\right) & f\left(x,y\right)
\end{smallmatrix}\right)
\end{align*}
\smallskip{}
\end{defn}

With this new form we rewrite \thmref{matrix-field-structure}:
\begin{thm}
\label{thm:normalized-matrix-field}Let $f,\bar{f},a,b$ be polynomials
as in \defref{conjugate}. We set 
\begin{align*}
M_{X}\left(x,y\right) & :=\tau U_{f\left(x,y\right)}U_{-\bar{f}\left(x+1,y\right)}D_{b\left(x+1\right)}=\left(\begin{smallmatrix}0 & 1\\
b\left(x+1\right) & f\left(x,y\right)-\bar{f}\left(x+1,y\right)
\end{smallmatrix}\right)\\
M_{Y}\left(x,y\right) & :=U_{f\left(x,y\right)}^{\tau}\tau D_{-\left(f\bar{f}\right)\left(0,y\right)}U_{\bar{f}\left(x,y\right)}^{\tau}=\left(\begin{smallmatrix}\bar{f}\left(x,y\right) & 1\\
b\left(x\right) & f\left(x,y\right)
\end{smallmatrix}\right)
\end{align*}
Then
\begin{enumerate}
\item The matrices form a conservative matrix field, namely
\[
M_{X}\left(x,y\right)M_{Y}\left(x+1,y\right)=M_{Y}\left(x,y\right)M_{X}\left(x,y+1\right).
\]
\item The determinants of $M_{X}\left(x,y\right),M_{Y}\left(x,y\right)$
are only functions of $x,y$ respectively, and more specifically:
\begin{align*}
\det\left(M_{X}\left(x,y\right)\right) & =-b_{X}\left(x+1\right)\\
\det\left(M_{Y}\left(x,y\right)\right) & =b_{Y}\left(y\right).
\end{align*}
\item For any $n\geq1,m\geq1$ we have 
\[
\tau U_{a\left(0,1\right)}\left[\prod_{k=1}^{n-1}M_{X}^{cf}\left(k,1\right)\cdot\prod_{k=1}^{m-1}M_{Y}^{cf}\left(n,k\right)\right]\left(0\right)=\left[\prod_{k=0}^{n-1}M_{X}\left(k,1\right)\cdot\prod_{k=1}^{m-1}M_{Y}\left(n,k\right)\right]\left(0\right)
\]
\end{enumerate}
\end{thm}

\begin{proof}
This follows directly from \thmref{matrix-field-structure} .
\end{proof}

\subsection{\label{subsec:The-dual-matrix-field}The dual conservative matrix
field}

In the previous section we saw that the matrices in our matrix field
are constructed similarly, up to a cyclic permutation:
\begin{align*}
M_{X}\left(x,y\right) & :=\tau U_{f\left(x,y\right)}U_{-\bar{f}\left(x+1,y\right)}D_{b\left(x+1\right)}=\left(\begin{smallmatrix}0 & 1\\
b\left(x+1\right) & f\left(x,y\right)-\bar{f}\left(x+1,y\right)
\end{smallmatrix}\right)\\
M_{Y}\left(x,y\right) & :=U_{f\left(x,y\right)}^{\tau}\tau D_{-\left(f\bar{f}\right)\left(0,y\right)}U_{\bar{f}\left(x,y\right)}^{\tau}=\left(\begin{smallmatrix}\bar{f}\left(x,y\right) & 1\\
b\left(x\right) & f\left(x,y\right)
\end{smallmatrix}\right)
\end{align*}
The next goal is to use this almost symmetry with the hope of eventually
saying something about the diagonal line $X=Y$.
\begin{defn}[\textbf{The dual matrix field}]
Let $f\left(x,y\right),\bar{f}\left(x,y\right)$ be conjugate polynomial,
and let $M_{X},M_{Y}$ be as above. We define the \textbf{dual matrix
field} to be
\begin{align*}
\hat{M}_{Y}\left(y,x\right) & =U_{\bar{f}\left(x-1,y\right)}^{\tau}M_{X}\left(x-1,y+1\right)U_{-\bar{f}\left(x,y\right)}^{\tau}=U_{f\left(x,y\right)}^{\tau}\tau D_{b_{X}\left(x\right)}U_{-\bar{f}\left(x,y\right)}^{\tau}\\
\hat{M}_{X}\left(y,x\right) & =U_{\bar{f}\left(x-1,y\right)}^{\tau}M_{Y}\left(x-1,y+1\right)U_{-\bar{f}\left(x-1,y+1\right)}^{\tau}=\tau U_{f\left(x,y\right)}U_{\bar{f}\left(x,y+1\right)}D_{-b_{Y}\left(y+1\right)}
\end{align*}
This new matrix field corresponds to the conjugate polynomials 
\begin{align*}
\hat{f}\left(x,y\right) & :=f\left(y,x\right)\\
\bar{\hat{f}}\left(x,y\right) & :=-\bar{f}\left(y,x\right)\\
\hat{a}\left(x,y\right) & :=\hat{f}\left(x,y\right)-\bar{\hat{f}}\left(x+1,y\right)=f\left(y,x\right)+\bar{f}\left(y,x+1\right)\\
\hat{b}_{X}\left(x\right) & :=-b_{Y}\left(x\right)
\end{align*}
\end{defn}

\begin{example}
In the $\zeta\left(3\right)$ matrix field mentioned in \exaref{zeta-3-matrix-field}
we have a special case where
\[
f\left(x,y\right)=\frac{y^{3}-x^{3}}{y-x}\left(y+x\right)\qquad;\qquad\bar{f}\left(x,y\right)=\frac{y^{3}+x^{3}}{y+x}\left(y-x\right),
\]
satisfy $f\left(x,y\right)=f\left(y,x\right)$ and $\bar{f}\left(x,y\right)=-\bar{f}\left(y,x\right)$,
so that $\hat{f}=f$ and $\bar{\hat{f}}=\bar{f}$.

In the one of the $\zeta\left(2\right)$ matrix field from \exaref{matrix-field-examples},
we have
\[
f\left(x,y\right)=2x^{2}+2xy+y^{2}\qquad;\qquad\bar{f}\left(x,y\right)=-2x^{2}+2xy-y^{2}
\]
so that 
\[
\hat{f}\left(x,y\right)=x^{2}+2xy+2y^{2}\qquad;\qquad\bar{\hat{f}}\left(x,y\right)=x^{2}-2xy+2y^{2}
\]
and therefore
\begin{align*}
\hat{a}\left(x,y\right) & =\left(x^{2}+2xy+2y^{2}\right)-\left(\left(x+1\right)^{2}-2\left(x+1\right)y+2y^{2}\right)=\left(2y-1\right)\left(2x+1\right)\\
\hat{b}_{X}\left(x\right) & =x^{4}.
\end{align*}
\end{example}

This dual matrix field construction not only gives us free of charge
another conservative matrix field for every one that we find, but
they are also closely related. In the matrix field with $M_{X},M_{Y}$
, the horizontal lines are (almost) polynomial continued fractions
which can be used to study the whole matrix field. By definition,
the horizontal lines of the dual matrix field correspond to vertical
lines in the original matrix field, so to understand the full matrix
field, we would want to understand these two families of  continued
fractions. 

More precisely, since
\[
M_{Y}\left(x,y\right)=U_{-\bar{f}\left(x,y-1\right)}^{\tau}\hat{M}_{X}\left(y-1,x+1\right)U_{\bar{f}\left(x,y\right)}^{\tau},
\]
we get that
\begin{equation}
\prod_{k=1}^{n}M_{Y}\left(y,k\right)=U_{-\bar{f}\left(y,0\right)}^{\tau}\left[\prod_{k=0}^{n-1}\hat{M}_{X}\left(k,y+1\right)\right]U_{\bar{f}\left(y,n\right)}^{\tau}\label{eq:dual-row-column}
\end{equation}

With this dualic structure we turn to study the rational approximations
given by the different points on the matrix field, and more concretely
how far the standard rational presentation is from being a reduced
rational presentation. 
\begin{defn}
\label{def:deno-nume}
\begin{enumerate}
\item For every pair of integers $n\geq0,\;m\geq1$ we let 
\[
\left(\begin{smallmatrix}P\left(n,m\right)\\
Q\left(n,m\right)
\end{smallmatrix}\right):=\left[\prod_{k=1}^{m-1}M_{Y}\left(0,k\right)\right]\left[\prod_{k=0}^{n-1}M_{X}\left(k,m\right)\right]e_{2}.
\]
\item For every $n\geq0$ define the polynomial vectors
\begin{align*}
\left(\begin{smallmatrix}p_{n}\left(y\right)\\
q_{n}\left(y\right)
\end{smallmatrix}\right) & =\left[\prod_{0}^{n-1}M_{X}\left(k,y\right)\right]e_{2}\\
\left(\begin{smallmatrix}\hat{p}_{n}\left(y\right)\\
\hat{q}_{n}\left(y\right)
\end{smallmatrix}\right) & =\left[\prod_{0}^{n-1}\hat{M}_{X}\left(k,y\right)\right]e_{2}.
\end{align*}
\end{enumerate}
In general we are interested in the function $\frac{P\left(n,m\right)}{Q\left(n,m\right)}$,
and its limit as $n,m\to\infty$. The numerators and denominators
$p_{n}\left(y\right),q_{n}\left(y\right)$ of the continued fractions
for a given line (and their duals), will be help us study that function. 
\end{defn}

For example, the first few values of $p_{n}\left(m\right),q_{n}\left(m\right)$
are arranged as :

\begin{align*}
\xyR{1pc}\xyC{1pc}\xymatrix{\left(*\right) &  &  & \left(*\right) &  &  & \left(*\right) &  &  & \left(*\right)\\
\\
{\left(\begin{smallmatrix}p_{0}\left(3\right)\\
q_{0}\left(3\right)
\end{smallmatrix}\right)}\ar[rrr]_{M_{X}\left(0,3\right)}\ar@{..>}[uu]^{M_{Y}\left(0,3\right)} &  &  & {\left(\begin{smallmatrix}p_{1}\left(3\right)\\
q_{1}\left(3\right)
\end{smallmatrix}\right)}\ar[rrr]_{M_{X}\left(1,3\right)}\ar@{..>}[uu]^{M_{Y}\left(1,3\right)} &  &  & {\left(\begin{smallmatrix}p_{2}\left(3\right)\\
q_{2}\left(3\right)
\end{smallmatrix}\right)}\ar[rrr]_{M_{X}\left(2,3\right)}\ar@{..>}[uu]^{M_{Y}\left(2,3\right)} &  &  & \left(*\right)\\
\\
{\left(\begin{smallmatrix}p_{0}\left(2\right)\\
q_{0}\left(2\right)
\end{smallmatrix}\right)}\ar[rrr]_{M_{X}\left(0,2\right)}\ar@{..>}[uu]^{M_{Y}\left(0,2\right)} &  &  & {\left(\begin{smallmatrix}p_{1}\left(2\right)\\
q_{1}\left(2\right)
\end{smallmatrix}\right)}\ar[rrr]_{M_{X}\left(1,2\right)}\ar@{..>}[uu]^{M_{Y}\left(1,2\right)} &  &  & {\left(\begin{smallmatrix}p_{2}\left(2\right)\\
q_{2}\left(2\right)
\end{smallmatrix}\right)}\ar[rrr]_{M_{X}\left(2,2\right)}\ar@{..>}[uu]^{M_{Y}\left(2,2\right)} &  &  & \left(*\right)\\
\\
{\left(\begin{smallmatrix}p_{0}\left(1\right)\\
q_{0}\left(1\right)
\end{smallmatrix}\right)}\ar[rrr]_{M_{X}\left(0,1\right)}\ar@{..>}[uu]^{M_{Y}\left(0,1\right)} &  &  & {\left(\begin{smallmatrix}p_{1}\left(1\right)\\
q_{1}\left(1\right)
\end{smallmatrix}\right)}\ar[rrr]_{M_{X}\left(1,1\right)}\ar@{..>}[uu]^{M_{Y}\left(1,1\right)} &  &  & {\left(\begin{smallmatrix}p_{2}\left(1\right)\\
q_{2}\left(1\right)
\end{smallmatrix}\right)}\ar[rrr]_{M_{X}\left(2,1\right)}\ar@{..>}[uu]^{M_{Y}\left(2,1\right)} &  &  & \left(*\right)
}
\end{align*}

\begin{rem}
Note that since $M_{X}\left(k,y\right)e_{1}=b\left(k+1\right)e_{2}$
and $\hat{M}_{X}\left(k,y\right)e_{1}=-b_{Y}\left(k\right)$, we have
for $n\geq1$
\begin{align*}
\left(\begin{smallmatrix}p_{n-1}\left(y\right) & p_{n}\left(y\right)\\
q_{n-1}\left(y\right) & q_{n}\left(y\right)
\end{smallmatrix}\right)D_{b_{X}\left(n\right)} & =\prod_{0}^{n-1}M_{X}\left(k,y\right)\\
\left(\begin{smallmatrix}\hat{p}_{n-1}\left(y\right) & \hat{p}_{n}\left(y\right)\\
\hat{q}_{n-1}\left(y\right) & \hat{q}_{n}\left(y\right)
\end{smallmatrix}\right)D_{-b_{Y}\left(n\right)} & =\prod_{0}^{n-1}\hat{M}_{X}\left(k,y\right).
\end{align*}
\end{rem}

\newpage{}

We now turn to connect between $P\left(n,m\right),Q\left(n,m\right)$
and $p_{n}\left(y\right),q_{n}\left(y\right),\hat{p}_{n}\left(y\right),\hat{q}_{n}\left(y\right)$,
so that eventually we can use this connection to study the limits
of $\frac{P\left(n,m\right)}{Q\left(n,m\right)}$ as $n,m\to\infty$.
\begin{claim}
\label{claim:Additive_form}Let $f,\bar{f}\in\ZZ\left[x,y\right]$
be conjugate polynomials.
\end{claim}

\begin{enumerate}
\item \label{enu:first-line-polynomials}If $b_{X}\left(x\right)=\fbf\left(x,0\right)$,
then
\[
U_{\bar{f}\left(0,0\right)}^{\tau}\left[\prod_{0}^{n-1}M_{X}\left(i,1\right)\right]U_{-\bar{f}\left(n,0\right)}^{\tau}=\prod_{1}^{n}\left(\begin{smallmatrix}-\bar{f}\left(i,0\right) & 1\\
0 & f\left(i,0\right)
\end{smallmatrix}\right).
\]
In particular we get that 
\[
\begin{pmatrix}p_{1}\left(n\right)\\
q_{1}\left(n\right)
\end{pmatrix}=U_{-\bar{f}\left(0,0\right)}^{\tau}\prod_{1}^{n}\left(\begin{smallmatrix}-\bar{f}\left(i,0\right) & 1\\
0 & f\left(i,0\right)
\end{smallmatrix}\right)e_{2}.
\]
\item \label{enu:reciprocal-polynomials}If $b_{X}\left(x\right)=\fbf\left(x,0\right)$
and $b_{Y}\left(y\right)=\fbf\left(0,y\right)$, then 
\[
\begin{pmatrix}P\left(n,m+1\right)\\
Q\left(n,m+1\right)
\end{pmatrix}=\left[\left(\begin{smallmatrix}\prod_{k=1}^{m}\bar{f}\left(0,k\right) & \hat{p}_{m}\left(1\right)\\
0 & \prod_{k=1}^{m}f\left(0,k\right)
\end{smallmatrix}\right)\right]\left(\begin{smallmatrix}p_{n}\left(m+1\right)\\
q_{n}\left(m+1\right)
\end{smallmatrix}\right)=U_{-\bar{f}\left(0,0\right)}^{\tau}\left(\begin{smallmatrix}\prod_{1}^{n}\left(-\bar{f}\right)\left(k,0\right) & p_{n}\left(1\right)\\
0 & \prod_{1}^{n}f\left(k,0\right)
\end{smallmatrix}\right)\left(\begin{smallmatrix}\hat{p}_{m}\left(n+1\right)\\
\hat{q}_{m}\left(n+1\right)
\end{smallmatrix}\right).
\]
In particular we have {\small{}
\[
\frac{P\left(n,m+1\right)}{Q\left(n,m+1\right)}=\prod_{k=1}^{m}\frac{\bar{f}\left(0,k\right)}{f\left(0,k\right)}\cdot\frac{p_{n}\left(m+1\right)}{q_{n}\left(m+1\right)}+\frac{\hat{p}_{m}\left(1\right)}{\prod_{k=1}^{m}f\left(0,k\right)}=U_{-\bar{f}\left(0,0\right)}^{\tau}\left(\left(-1\right)^{n}\prod_{k=1}^{n}\frac{\bar{f}\left(k,0\right)}{f\left(k,0\right)}\cdot\frac{\hat{p}_{m}\left(n+1\right)}{\hat{q}_{m}\left(n+1\right)}+\frac{p_{n}\left(1\right)}{\prod_{k=1}^{n}f\left(k,0\right)}\right).
\]
}{\small\par}
\end{enumerate}
\begin{proof}
\begin{enumerate}
\item Given that $b_{X}\left(x\right)=\fbf\left(x,0\right)$ we have that
\begin{align*}
U_{\bar{f}\left(x,0\right)}^{\tau}M_{X}\left(x,1\right)U_{-\bar{f}\left(x+1,0\right)}^{\tau} & =\left(\begin{smallmatrix}1 & 0\\
\bar{f}\left(x,0\right) & 1
\end{smallmatrix}\right)\left(\begin{smallmatrix}0 & 1\\
\left(f\bar{f}\right)\left(x+1,0\right) & f\left(x+1,0\right)-\bar{f}\left(x,0\right)
\end{smallmatrix}\right)\left(\begin{smallmatrix}1 & 0\\
-\bar{f}\left(x+1,0\right) & 1
\end{smallmatrix}\right)=\left(\begin{smallmatrix}-\bar{f}\left(x+1,0\right) & 1\\
0 & f\left(x+1,0\right)
\end{smallmatrix}\right),
\end{align*}
implying that 
\[
U_{\bar{f}\left(0,0\right)}^{\tau}\left[\prod_{0}^{n-1}M_{X}\left(i,1\right)\right]U_{-\bar{f}\left(n,0\right)}^{\tau}=\prod_{1}^{n}\left(\begin{smallmatrix}-\bar{f}\left(i,0\right) & 1\\
0 & f\left(i,0\right)
\end{smallmatrix}\right).
\]
Using the fact that $U_{-\bar{f}\left(n,0\right)}^{\tau}e_{2}=e_{2}$,
we conclude part (1) of this claim.
\item We compute $\frac{P\left(n,m+1\right)}{Q\left(n,m+1\right)}$ in two
different ways - first by moving in the matrix field along the $X=0$
line and then $Y=m+1$ line, and second by moving along the $Y=1$
line and then the $X=n$ line, namely 
\begin{equation}
\prod_{1}^{m}M_{Y}\left(0,k\right)\prod_{0}^{n-1}M_{X}\left(k,m+1\right)e_{2}=\prod_{0}^{n-1}M_{X}\left(k,1\right)\prod_{1}^{m}M_{Y}\left(n,k\right)e_{2}.\label{eq:two_path_potential}
\end{equation}
Dealing first with the left hand side, by definition we have that
\[
\prod_{0}^{n-1}M_{X}\left(k,m+1\right)e_{2}=\left(\begin{smallmatrix}p_{n}\left(m+1\right)\\
q_{n}\left(m+1\right)
\end{smallmatrix}\right).
\]
We now claim that 
\[
\prod_{1}^{m}M_{Y}\left(0,k\right)=\begin{pmatrix}\prod_{1}^{m}\bar{f}\left(0,k\right) & \hat{p}_{m}\left(1\right)\\
0 & \prod_{1}^{m}f\left(0,k\right)
\end{pmatrix}.
\]
To show that, note first that the condition on $b_{X},b_{Y}$ together
with the quadratic condition of conjugate polynomials implies that
\[
b_{X}\left(x\right)+b_{Y}\left(y\right)=\fbf\left(x,y\right)=\fbf\left(x,0\right)+\fbf\left(0,y\right)-\fbf\left(0,0\right)=b_{X}\left(x\right)+b_{Y}\left(y\right)-\fbf\left(0,0\right).
\]
It then follows that $b_{X}\left(0\right)=b_{Y}\left(0\right)=\fbf\left(0,0\right)=0$,
and therefore $M_{Y}\left(0,k\right)=\left(\begin{smallmatrix}\bar{f}\left(0,k\right) & 1\\
0 & f\left(0,k\right)
\end{smallmatrix}\right)$ is upper triangular. Thus, in the product $\prod_{1}^{m}M_{Y}\left(0,k\right)$,
the main issue is to find the upper right coordinate. For that, we
use \eqref{dual-row-column} to obtain 
\[
e_{1}^{tr}\prod_{1}^{m}M_{Y}\left(0,k\right)e_{2}=e_{1}^{tr}U_{-\bar{f}\left(0,0\right)}^{\tau}\left[\prod_{k=0}^{m-1}\hat{M}_{X}\left(k,1\right)\right]U_{\bar{f}\left(0,m\right)}^{\tau}e_{2}=e_{1}^{tr}\left[\prod_{k=0}^{m-1}\hat{M}_{X}\left(k,1\right)\right]e_{2}=\hat{p}_{m}\left(1\right),
\]
which produces the left side of the equation in \eqref{two_path_potential}.
\\
For the right hand side, we use \eqref{dual-row-column} and part
(1) of this claim to obtain
\begin{align*}
\left[\prod_{0}^{n-1}M_{X}\left(k,1\right)\right]\left[\prod_{1}^{m}M_{Y}\left(n,k\right)\right]e_{2} & =\left[U_{-\bar{f}\left(0,0\right)}^{\tau}\prod_{1}^{n}\left(\begin{smallmatrix}-\bar{f}\left(k,0\right) & 1\\
0 & f\left(k,0\right)
\end{smallmatrix}\right)\right]\cdot\left[\prod_{k=0}^{m-1}\hat{M}_{X}\left(k,n+1\right)U_{\bar{f}\left(y,m\right)}^{\tau}\right]e_{2}\\
 & =U_{-\bar{f}\left(0,0\right)}^{\tau}\prod_{1}^{n}\left(\begin{smallmatrix}-\bar{f}\left(i,0\right) & 1\\
0 & f\left(i,0\right)
\end{smallmatrix}\right)\left(\begin{smallmatrix}\hat{p}_{m}\left(n+1\right)\\
\hat{q}_{m}\left(n+1\right)
\end{smallmatrix}\right)
\end{align*}
Similarly to the previous case, and using part (1) of this claim again,
we conclude that the last expression equals to
\[
U_{-\bar{f}\left(0,0\right)}^{\tau}\left(\begin{smallmatrix}\left(-1\right)^{n}\prod_{1}^{n}\bar{f}\left(i,0\right) & p_{n}\left(1\right)\\
0 & \prod_{1}^{n}f\left(i,0\right)
\end{smallmatrix}\right)\left(\begin{smallmatrix}\hat{p}_{m}\left(n+1\right)\\
\hat{q}_{m}\left(n+1\right)
\end{smallmatrix}\right),
\]
thus completing the proof.
\end{enumerate}
\end{proof}
\begin{rem}
In the case where $\bar{f}\left(0,0\right)=0$, we get that $U_{f\left(0,0\right)}=Id$,
making the results above clearer. In particular, part (1) shows that
$q_{n}\left(1\right)=\prod_{1}^{n}f\left(k,0\right)$, and then part
(2) implies that 
\[
Q\left(n,m+1\right)=\hat{q}_{m}\left(1\right)q_{n}\left(m+1\right)=q_{n}\left(1\right)\hat{q}_{m}\left(n+1\right).
\]
This will become more helpful later (see \lemref{q-n-lemma}).
\end{rem}

The last result is already enough to show some properties of the limit
of $\frac{P\left(n,m+1\right)}{Q\left(n,m+1\right)}$ as $n\to\infty$.
\begin{thm}
\label{thm:limit_on_each_line}Let $f,\bar{f}\in\ZZ\left[x,y\right]$
be conjugate polynomials such that $b_{X}\left(x\right)=\fbf\left(x,0\right)$
and $b_{Y}\left(y\right)=\fbf\left(0,y\right)$, and set $d=\deg_{y}\left(\hat{a}\left(x,y\right)\right)=\deg_{y}\left(f\left(y,x\right)+\bar{f}\left(y,x+1\right)\right)$.
If $\left|\prod_{k=1}^{n}\frac{\bar{f}\left(k,0\right)}{f\left(k,0\right)}\right|=o\left(n^{d}\right)$,
then 
\[
L\left(m\right)=\limfi n\frac{P\left(n,m\right)}{Q\left(n,m\right)}
\]
is independent of $m$.
\end{thm}

\begin{proof}
Using part \enuref{reciprocal-polynomials} in \claimref{dual-field-identities},
we obtain the equality 
\begin{align*}
U_{\bar{f}\left(0,0\right)}^{\tau}\left(\frac{P\left(n,m\right)}{Q\left(n,m\right)}\right) & =\frac{\hat{p}_{m}\left(n+1\right)}{\hat{q}_{m}\left(n+1\right)}\cdot\left(-1\right)^{n}\prod_{k=1}^{n}\frac{\bar{f}\left(k,0\right)}{f\left(k,0\right)}+\frac{p_{n}\left(1\right)}{\prod_{k=1}^{m}f\left(k,0\right)}.
\end{align*}
Under the assumption that $\left|\prod_{k=1}^{n}\frac{\bar{f}\left(k,0\right)}{f\left(k,0\right)}\right|=o\left(n^{d}\right)$,
it is enough to show that $\limfi n\frac{\hat{p}_{m}\left(n+1\right)}{\hat{q}_{m}\left(n+1\right)}=O\left(\frac{1}{n^{d}}\right)$
for all $m$, and this will follow if we can show that $\deg\left(\hat{p}_{m}\right)+d\leq\deg\left(\hat{q}_{m}\right)$
for all $m$.

Indeed, by definition we have that 
\[
\left(\begin{smallmatrix}\hat{p}_{m}\left(y\right)\\
\hat{q}_{m}\left(y\right)
\end{smallmatrix}\right)=\left[\prod_{0}^{m-1}\hat{M}_{X}\left(k,y\right)\right]e_{2}=\left[\prod_{k=0}^{m-1}\begin{pmatrix}0 & 1\\
-\left(f\bar{f}\right)\left(0,k+1\right) & f\left(y,k\right)+\bar{f}\left(y,k+1\right)
\end{pmatrix}\right]e_{2}.
\]
Under the assumption  that $d=\deg_{y}\left(f\left(y,k\right)+\bar{f}\left(y,k+1\right)\right)$,
a simple induction shows 
\[
\deg_{y}\left(\hat{p}_{m}\left(y\right)\right)=d\left(m-1\right)<dm=\deg_{y}\left(\hat{q}_{m}\left(y\right)\right),
\]
which completes the proof.
\end{proof}
\begin{rem}
Note that $\left|\prod_{k=1}^{n}\frac{\bar{f}\left(k,0\right)}{f\left(k,0\right)}\right|=o\left(n^{d}\right)$
automatically holds (resp. fails) if $\limfi x\left|\frac{\bar{f}\left(x,0\right)}{f\left(x,0\right)}\right|<1$
(resp. $>1$). Thus, to fully understand the asymptotic, we only need
to understand the case where $\limfi x\left|\frac{\bar{f}\left(x,0\right)}{f\left(x,0\right)}\right|=1$.
Since we only use the absolute value of the ratio, we may assume that
both $f,\bar{f}$ are monic of the same degree.
\end{rem}

\begin{claim}
Let $F\left(x\right),\bar{F}\left(x\right)$ be two distinct monic
real polynomials of the same degree.
\begin{align*}
F\left(x\right) & =x^{m}+\sum_{0}^{m-1}a_{i}x^{i}\\
\bar{F}\left(x\right) & =x^{m}+\sum_{0}^{m-1}\bar{a}_{i}x^{i}.
\end{align*}
\begin{enumerate}
\item If $a_{m-1}=\overline{a}_{m-1}$, then $\left|\prod_{1}^{n}\frac{\overline{F}\left(k\right)}{F\left(k\right)}\right|=O\left(1\right)$
is bounded.
\item If $a_{m-1}\neq\overline{a}_{m-1}$, then for any $\varepsilon>0$
we have that 
\[
\left|\prod_{1}^{n}\frac{\overline{F}\left(k\right)}{F\left(k\right)}\right|=O\left(n^{\bar{a}_{m-1}-a_{m-1}+\varepsilon}\right).
\]
In particular, if $\overline{a}_{m-1}-a_{m-1}<d$, then $\left|\prod_{1}^{n}\frac{\overline{F}\left(k\right)}{F\left(k\right)}\right|=o\left(n^{d}\right)$.
\end{enumerate}
\end{claim}

\begin{proof}
We first note that we may assume that $\frac{\overline{F}\left(x\right)}{F\left(x\right)}>0$
for all $x\geq0$. Indeed, since $\limfi x\frac{\overline{F}\left(x\right)}{F\left(x\right)}=1$,
we can choose $N$ large enough such that $G\left(x\right):=F\left(x+N\right),\;\overline{G}\left(x\right):=F\left(x+N\right)$
satisfy $\frac{\overline{G}\left(x\right)}{G\left(x\right)}>0$ for
all $x\geq0$. Moving from $F,\bar{F}$ to $G,\bar{G}$ doesn't change
the difference of $\overline{a}_{m-1}-a_{m-1}$, and since $\left|\prod_{1}^{n}\frac{\overline{F}\left(k\right)}{F\left(k\right)}\right|=\left|\prod_{1}^{N}\frac{\overline{F}\left(k\right)}{F\left(k\right)}\right|\left|\prod_{1}^{n-N}\frac{\overline{G}\left(k\right)}{G\left(k\right)}\right|$,
the asymptotic of $\left|\prod_{1}^{n}\frac{\overline{G}\left(k\right)}{G\left(k\right)}\right|$
gives the needed upper bound on the asymptotics of $\left|\prod_{1}^{n}\frac{\overline{F}\left(k\right)}{F\left(k\right)}\right|$.

Now, with this assumption that $\frac{\overline{F}\left(x\right)}{F\left(x\right)}>0$
for all $x\geq0$, we can look at the well defined expression $\ln\left(\frac{\overline{F}\left(x\right)}{F\left(x\right)}\right)$
for $x>0$. Let $j$ be the largest index such that $a_{j}\neq\overline{a}_{j}$,
and define $A\left(x\right)$ so that
\[
\frac{\bar{F}\left(x\right)}{F\left(x\right)}-1=\frac{\sum_{0}^{j}\left(\bar{a}_{i}-a_{i}\right)x^{i}}{x^{m}+\sum_{0}^{m-1}a_{i}x^{i}}=\frac{\left(\bar{a}_{j}-a_{j}\right)}{x^{m-j}}\overbrace{\left(\frac{1+\sum_{i=0}^{j-1}\frac{\left(\bar{a}_{i}-a_{i}\right)}{\left(\bar{a}_{j}-a_{j}\right)}x^{i-j}}{1+\sum_{i=0}^{m-1}a_{i}x^{i-m}}\right)}^{A\left(x\right)}.
\]
It follows that $\limfi xA\left(x\right)=1$, and
\begin{align*}
\ln\left(\prod_{k=1}^{n}\frac{\overline{F}\left(k\right)}{F\left(k\right)}\right)=\sum_{k=1}^{n}\ln\left(\frac{\bar{F}\left(x\right)}{F\left(x\right)}\right) & =\sum_{k=1}^{n}\ln\left(1+\frac{\left(\bar{a}_{j}-a_{j}\right)A\left(k\right)}{k^{m-j}}\right)\leq\left(\bar{a}_{j}-a_{j}\right)\sum_{k=1}^{n}\frac{A\left(k\right)}{k^{m-j}}.
\end{align*}
Using the fact that $\limfi xA\left(x\right)=1$, we see that for
any $\varepsilon'>0$, we can find some constant $C_{\varepsilon'}$
such that 
\[
-C_{\varepsilon'}+\left(1-\varepsilon'\right)\sum_{k=1}^{n}\frac{1}{k^{m-j}}\leq\sum_{k=1}^{n}\frac{A\left(k\right)}{k^{m-j}}\leq C_{\varepsilon'}+\left(1+\varepsilon'\right)\sum_{k=1}^{n}\frac{1}{k^{m-j}}.
\]
If $m-j\geq2$, then both the upper and lower bounds converge to a
finite limit, so that sequence in the middle is bounded, and therefore
so is $\prod_{k=1}^{n}\frac{\overline{F}\left(k\right)}{F\left(k\right)}$.

If $m-j=1$, then by changing the harmonic sum $\sum_{k=1}^{n}\frac{1}{k}$
into $\ln\left(n\right)$, and adjusting the constant accordingly,
we get that 
\[
-C_{\varepsilon'}+\left(1-\varepsilon'\right)\ln\left(n\right)\leq\sum_{k=1}^{n}\frac{A\left(k\right)}{k}\leq C_{\varepsilon'}+\left(1+\varepsilon'\right)\ln\left(n\right).
\]
Multiplying the inequalities above by $\left|\overline{a}_{m-1}-a_{m-1}\right|$
we get that 
\[
-C_{\varepsilon}+\left(\left|\overline{a}_{m-1}-a_{m-1}\right|-\varepsilon\right)\ln\left(n\right)\leq\left|\overline{a}_{m-1}-a_{m-1}\right|\sum_{k=1}^{n}\frac{A\left(k\right)}{k^{m-j}}\leq C_{\varepsilon}+\left(\left|\overline{a}_{m-1}-a_{m-1}\right|+\varepsilon\right)\ln\left(n\right),
\]
where $C_{\varepsilon}=\left|\overline{a}_{m-1}-a_{m-1}\right|C_{\varepsilon'}$
and $\varepsilon=\varepsilon'\cdot\left|\overline{a}_{m-1}-a_{m-1}\right|$.
It then follows that 
\[
\prod_{k=1}^{n}\frac{\overline{F}\left(k\right)}{F\left(k\right)}\leq\exp\left(\left(\bar{a}_{j}-a_{j}\right)\sum_{k=1}^{n}\frac{A\left(k\right)}{k^{m-j}}\right)\leq e^{C_{\varepsilon}}n^{\bar{a}_{j}-a_{j}+\varepsilon},
\]
which completes the proof.
\end{proof}
\begin{example}
Consider two of the examples from \exaref{matrix-field-examples}
(and in \secref{The-z3-case} we will apply this for $\zeta\left(3\right)$).
\begin{enumerate}
\item For $f\left(x,y\right)=x+y$ and $\bar{f}\left(x,y\right)=x-y$ the
limit on the line $y=1$ was $\frac{1-\ln\left(2\right)}{\ln\left(2\right)}$.
Since $\frac{\bar{f}\left(k,0\right)}{f\left(k,0\right)}=1$ for all
$k\geq1$ and 
\[
d=\deg_{y}\left(f\left(y,x\right)+\bar{f}\left(y,x+1\right)\right)=\deg_{y}\left(y+x+y-x-1\right)=\deg_{y}\left(2y-1\right)=1,
\]
the conditions of the last theorem holds, so that $\limfi n\frac{P\left(n,m\right)}{Q\left(n,m\right)}=\frac{1-\ln\left(2\right)}{\ln\left(2\right)}$
for all $m$ on the corresponding matrix field.
\item Taking $f\left(x,y\right)=x+y$ and $\bar{f}\left(x,y\right)=1$,
we have that $\left|\frac{\bar{f}\left(k,0\right)}{f\left(k,0\right)}\right|=\frac{1}{k}=o\left(1\right)$
which is automatically $o\left(n^{d}\right)$ for all $d$. Thus $\limfi n\frac{P\left(n,m\right)}{Q\left(n,m\right)}=\limfi n\frac{P\left(n,1\right)}{Q\left(n,1\right)}=\frac{1}{e-1}$.
\end{enumerate}
\end{example}

On \thmref{limit_on_each_line} we only considered the limits $\limfi n\frac{P\left(n,m\right)}{Q\left(n,m\right)}$
for fixed $m$. The general case is much more complicated, and it
is not clear if there is a simple condition to show convergence. In
\secref{The-z3-case}, we will prove for the matrix field of $\zeta\left(3\right)$
that the limit is always $\zeta\left(3\right)$ no matter the path
chosen. For that we will need the following:
\begin{claim}
\label{claim:dual-field-identities}Let $f,\bar{f}\in\ZZ\left[x,y\right]$
be conjugate polynomials and let $p_{n},q_{n}$ as in \defref{deno-nume}
above. Then
\begin{enumerate}
\item \label{enu:Increase-n}\textbf{\uline{Horizontal lines}}\textbf{:
}When increasing $n$ we get 3-term recurrence
\begin{align*}
\left(\begin{smallmatrix}p_{n+1}\left(y\right)\\
q_{n+1}\left(y\right)
\end{smallmatrix}\right) & =\left(\begin{smallmatrix}p_{n-1}\left(y\right) & p_{n}\left(y\right)\\
q_{n-1}\left(y\right) & q_{n}\left(y\right)
\end{smallmatrix}\right)\left(\begin{smallmatrix}b_{X}\left(n\right)\\
a\left(n,y\right)
\end{smallmatrix}\right).
\end{align*}
\item \label{enu:increase-m}\textbf{\uline{Vertical lines}}\textbf{:}
When $\left(f\bar{f}\right)\left(0,m\right)\neq0$, increasing $m$
follows the recurrence 
\begin{align*}
\left(\begin{smallmatrix}p_{n}\left(m+1\right)\\
q_{n}\left(m+1\right)
\end{smallmatrix}\right) & =\frac{1}{\left(f\bar{f}\right)\left(0,m\right)}\left(\begin{smallmatrix}f\left(0,m\right) & -1\\
0 & \bar{f}\left(0,m\right)
\end{smallmatrix}\right)\left(\begin{smallmatrix}p_{n-1}\left(m\right) & p_{n}\left(m\right)\\
q_{n-1}\left(m\right) & q_{n}\left(m\right)
\end{smallmatrix}\right)\left(\begin{smallmatrix}b_{X}\left(n\right)\\
f\left(n,m\right)
\end{smallmatrix}\right)
\end{align*}
\end{enumerate}
\end{claim}

\begin{proof}
\begin{enumerate}
\item Follows directly from the equation $\prod_{0}^{n}M_{X}^{cf}\left(k,y\right)=\left(\begin{smallmatrix}p_{n-1}\left(m\right) & p_{n}\left(m\right)\\
q_{n-1}\left(m\right) & q_{n}\left(m\right)
\end{smallmatrix}\right)$ in \lemref{gcf-recursion}. 
\item This follows from the conservativeness condition of the matrix field
structure
\begin{align*}
\left(\begin{smallmatrix}p_{n}\left(m+1\right)\\
q_{n}\left(m+1\right)
\end{smallmatrix}\right) & =\left[\prod_{0}^{n-1}M_{X}\left(k,m+1\right)\right]e_{2}=M_{Y}\left(0,m\right)^{-1}\left[\prod_{0}^{n-1}M_{X}\left(k,m\right)\right]M_{Y}\left(n,m\right)e_{2}\\
 & =\left(\begin{smallmatrix}\bar{f}\left(0,m\right) & 1\\
0 & f\left(0,m\right)
\end{smallmatrix}\right)^{-1}\left(\begin{smallmatrix}p_{n-1}\left(m\right) & p_{n}\left(m\right)\\
q_{n-1}\left(m\right) & q_{n}\left(m\right)
\end{smallmatrix}\right)D_{b_{X}\left(n\right)}\cdot\left(\begin{smallmatrix}1\\
f\left(n,m\right)
\end{smallmatrix}\right)\\
 & =\frac{1}{\left(f\bar{f}\right)\left(0,m\right)}\left(\begin{smallmatrix}f\left(0,m\right) & -1\\
0 & \bar{f}\left(0,m\right)
\end{smallmatrix}\right)\left(\begin{smallmatrix}p_{n-1}\left(m\right) & p_{n}\left(m\right)\\
q_{n-1}\left(m\right) & q_{n}\left(m\right)
\end{smallmatrix}\right)\left(\begin{smallmatrix}b_{X}\left(n\right)\\
f\left(n,m\right)
\end{smallmatrix}\right)
\end{align*}
\end{enumerate}
\end{proof}

\newpage{}

\section{\label{sec:The-z3-case}The $\zeta\left(3\right)$ case}

We now apply the dual matrix field identities for the $\zeta\left(3\right)$
matrix field. Recall that in this case we have that 
\begin{align*}
f\left(x,y\right) & =\frac{y^{3}-x^{3}}{y-x}\left(y+x\right)=y^{3}+2y^{2}x+2yx^{2}+x^{3}\\
\bar{f}\left(x,y\right) & =\frac{y^{3}+x^{3}}{y+x}\left(y-x\right)=y^{3}-2y^{2}x+2yx^{2}-x^{3}\\
a\left(x,y\right) & =x^{3}+\left(1+x\right)^{3}+2y\left(y-1\right)\left(2x+1\right)\\
\fbf\left(x,y\right) & =y^{6}-x^{6}\\
b_{X}\left(x\right) & =-x^{6},\quad b_{Y}\left(y\right)=y^{6}
\end{align*}

Since $b_{X}\left(x\right)=\fbf\left(x,0\right)$, the bottom line
of this matrix field can be easily converted to infinite sum via Euler
method and more specifically part \enuref{first-line-polynomials}
in \claimref{Additive_form} shows that 
\begin{equation}
\begin{pmatrix}p_{1}\left(n\right)\\
q_{1}\left(n\right)
\end{pmatrix}=\prod_{1}^{n}\left(\begin{smallmatrix}k^{3} & 1\\
0 & k^{3}
\end{smallmatrix}\right)e_{2}=\left(\begin{smallmatrix}\left(n!\right)^{3} & \left(n!\right)^{3}\sum_{k=1}^{n}\frac{1}{k^{3}}\\
0 & \left(n!\right)^{3}
\end{smallmatrix}\right)e_{2}\quad\Rightarrow\quad\frac{p_{1}\left(n\right)}{q_{1}\left(n\right)}=\sum_{k=1}^{n}\frac{1}{k^{3}}.\label{eq:zeta_3_bottom_line}
\end{equation}

Our first goal is showing that $\gcd\left(q_{n}\left(m\right),p_{n}\left(m\right)\right)$
is large as $n$ and $m$ increase, and we will use the results from
\claimref{Additive_form} and \claimref{dual-field-identities}. Once
we understand these polynomials and their greatest common divisor,
which are defined for each row separately, we will combine them together
to understand the general numerators and denominators appearing in
any route on the matrix field, starting at the bottom left corner.
In particular, investigating the diagonal route, we will show that
both the approximations converge fast enough, and the greatest common
divisor grows fast enough to conclude at the end that $\zeta\left(3\right)$
is irrational.

This $\zeta\left(3\right)$ matrix field has several properties making
it easier to work with, which will come into play later:
\begin{fact}
\label{fact:zeta_3}
\begin{enumerate}
\item The matrix field is its own dual, since $f\left(y,x\right)=f\left(x,y\right)$
and $-\bar{f}\left(y,x\right)=\bar{f}\left(x,y\right)$. In particular
we get that $p=\hat{p}$ and $q=\hat{q}$.
\item We have that $b_{X}\left(x\right)=\fbf\left(x,0\right)$, $b_{Y}\left(y\right)=\fbf\left(0,y\right)$,
and $f\left(0,0\right)=\bar{f}\left(0,0\right)=0$.
\item All the $f\left(n,0\right)=f\left(0,n\right)=\bar{f}\left(0,n\right)=-\bar{f}\left(n,0\right)=n^{3}$
are the same up to a sign (and therefore also $\hat{f}$ and $\bar{\hat{f}}$).
Furthermore, they all divide $f\left(n,n\right)=\hat{f}\left(n,n\right)=6n^{3}$.
\item We can write $a\left(x,y\right)$ as $a\left(x,y\right)=A_{1}\left(x\right)+y\left(y-1\right)A_{2}\left(x\right)$,
so in particular $a\left(x,1-y\right)=a\left(x,y\right)$.
\end{enumerate}
\end{fact}

Next, we want to show that $\gcd\left(q_{n}\left(m\right),p_{n}\left(m\right)\right)$,
is almost divisible by $\left(n!\right)^{3}$. We already know that
fixing $m$ and only increasing $n$, namely running on horizontal
lines in the matrix field, we get ``nice'' continued fractions which
should have factorial reduction. The next lemma shows that these factorial
reductions are in a sense synchronized between the different horizontal
lines.

In the following, we will use \textbf{$\lcm\left[n\right]$ }for $\lcm\left\{ 1,2,...,n\right\} $
where $n\geq1$ and also set $\lcm\left[0\right]=1$.
\begin{lem}
\label{lem:q-n-lemma}For all $n\geq0$ and $m\in\ZZ$ we have $\left(n!\right)^{3}\mid q_{n}\left(m\right)$
(with equality for $m=1$) and $\left(\frac{n!}{\lcm\left[n\right]}\right)^{3}\mid p_{n}\left(m\right)$.
In particular we have that $\left(\frac{n!}{\lcm\left[n\right]}\right)^{3}\mid\gcd\left(p_{n}\left(m\right),q_{n}\left(m\right)\right)$.
\end{lem}

\begin{proof}
We will prove this claim by induction on $n$, but before that, as
we mentioned above, for the $m=1$ case (the bottom horizontal line)
we get that 
\begin{align*}
p_{n}\left(1\right) & =\sum_{1}^{n}\left(\frac{n!}{k}\right)^{3}\\
q_{n}\left(1\right) & =\prod_{1}^{n}f\left(k,0\right)=\left(n!\right)^{3}.
\end{align*}
These are exactly the numerator and denominator of $\sum_{1}^{n}\frac{1}{k^{3}}$
when taking the product of the denominators. Since we can also instead
take the least common multiple of the denominators, we see that $\left(\frac{n!}{\lcm\left[n\right]}\right)^{3}\mid p_{n}\left(1\right)$
as required. Of course the $\left(n!\right)^{3}\mid q_{n}\left(1\right)$
is trivial since $\left(n!\right)^{3}=q_{n}\left(1\right)$, but more
over it allows us to think of the general conditions as $q_{n}\left(1\right)\mid q_{n}\left(m\right)$
and $q_{1}\left(n\right)\mid\lcm\left[n\right]^{3}p_{n}\left(m\right)$.\\

We prove the rest of this lemma using induction on $n$. The induction
hypothesis will go as follows - assuming that the claim is true for
$\left(n-1,m\right)$ for a given $n$ and all $m\in\ZZ$, we show:
\begin{enumerate}
\item From part \enuref{reciprocal-polynomials} in \claimref{Additive_form},
we show that the claim is true for $\left(n,m\right)$ with $1\leq m\leq n$.
\item From part \enuref{increase-m} in \claimref{dual-field-identities},
if the claim is true for $\left(n-1,n\right)$ and $\left(n,n\right)$,
then it is true for $\left(n,n+1\right)$.
\item Our polynomials satisfy $q_{n}\left(y\right)=q_{n}\left(1-y\right)$
and $p_{n}\left(y\right)=p_{n}\left(1-y\right)$, so the claim is
true for $\left(n,m\right)$ with $-n\leq m\leq n+1$, which are $2\left(n+1\right)$
consecutive integers.
\item These polynomials have degree $\leq2n+1$, so this is enough to show
the claim for $\left(n,m\right)$ for all $m$.
\end{enumerate}

When $n=0$ we have $q_{0}\left(m\right)\equiv1$ and $p_{0}\left(m\right)\equiv0$
which are divisible by $\left(0!\right)^{3}=1$ and $\frac{0!}{\lcm\left[0\right]}=1$
respectively.

Suppose now that the claim is true for $\left(k,m\right)$ with $k\leq n-1$
and all $m$ and we prove for $\left(n,m\right)$ and all $m$. We
prove first for the denominators, which is easier.\\

\textbf{\uline{Denominators:}}

By using identities from \claimref{Additive_form}, together with
the facts in \factref{zeta_3} about the matrix field we get 
\[
\left(\begin{smallmatrix}q_{m}\left(1\right) & p_{m}\left(1\right)\\
0 & q_{m}\left(1\right)
\end{smallmatrix}\right)\left(\begin{smallmatrix}p_{n}\left(m+1\right)\\
q_{n}\left(m+1\right)
\end{smallmatrix}\right)=\left(\begin{smallmatrix}q_{n}\left(1\right) & p_{n}\left(1\right)\\
0 & q_{n}\left(1\right)
\end{smallmatrix}\right)\left(\begin{smallmatrix}p_{m}\left(n+1\right)\\
q_{m}\left(n+1\right)
\end{smallmatrix}\right).
\]
For the denominators, this implies that 
\[
\frac{q_{n}\left(m+1\right)}{q_{n}\left(1\right)}=\frac{q_{m}\left(n+1\right)}{q_{m}\left(1\right)}.
\]
By the induction hypothesis, for $0\leq m\leq n-1$ the right hand
of this equation is an integer, so that $q_{n}\left(1\right)\mid q_{n}\left(m+1\right)$.
Using part \enuref{increase-m} in \claimref{dual-field-identities}
with $n=m$ we have

\begin{align*}
\left(\begin{smallmatrix}p_{n}\left(n+1\right)\\
q_{n}\left(n+1\right)
\end{smallmatrix}\right) & =\frac{1}{\left(f\bar{f}\right)\left(0,n\right)}\left(\begin{smallmatrix}f\left(0,n\right) & -1\\
0 & \bar{f}\left(0,n\right)
\end{smallmatrix}\right)\left(\begin{smallmatrix}p_{n-1}\left(n\right) & p_{n}\left(n\right)\\
q_{n-1}\left(n\right) & q_{n}\left(n\right)
\end{smallmatrix}\right)\left(\begin{smallmatrix}\left(f\bar{f}\right)\left(n,0\right)\\
f\left(n,n\right)
\end{smallmatrix}\right)
\end{align*}
so for the denominators we get
\[
q_{n}\left(n+1\right)=q_{n-1}\left(n\right)f\left(n,0\right)\frac{\bar{f}\left(n,0\right)}{f\left(0,n\right)}+q_{n}\left(n\right)\frac{f\left(n,n\right)}{f\left(0,n\right)}=q_{n-1}\left(n\right)n^{3}\left(-1\right)+q_{n}\left(n\right)\cdot6.
\]
By the induction hypothesis $\left(n-1\right)!^{3}\mid q_{n-1}\left(n\right)$
and from the argument above $n!^{3}\mid q_{n}\left(n\right)$, so
we conclude that $n!^{3}\mid q_{n}\left(n+1\right)$. At this point,
we know the claim for $\left(n,m\right)$ with $1\leq m\leq n+1$.

Using the fact that $a\left(x,y\right)$ can be written as $A_{1}\left(x\right)+y\left(y-1\right)A_{2}\left(x\right)$
, we get that $a\left(x,y\right)=a\left(x,1-y\right)$. Since
\begin{align*}
q_{n}\left(y\right) & =e_{2}^{tr}\left[\prod_{0}^{n-1}M_{X}\left(k,y\right)\right]e_{2}=e_{2}^{tr}\left[\prod_{0}^{n-1}\left(\begin{smallmatrix}0 & 1\\
b\left(k+1\right) & a\left(k,y\right)
\end{smallmatrix}\right)\right]e_{2}
\end{align*}
we also get that $q_{n}\left(1-y\right)=q_{n}\left(y\right)$ and
$\deg_{y}\left(q_{n}\right)\leq2n$. From this we conclude that $q_{n}\left(1\right)\mid q_{n}\left(m\right)$
for all $-n\leq m\leq n+1$, which is a total of $2n+2\geq\deg_{y}\left(q_{n}\right)+1$
consecutive integers. Finally, using \lemref{binomial-polynomials}
from \appref{integer-values} about integer valued polynomials, we
conclude that $q_{n}\left(1\right)\mid q_{n}\left(m\right)$ for all
$m$, thus proving the induction step for the denominators.\\
\newpage{}

\textbf{\uline{Numerators:}}

The proof for the numerators that $q_{1}\left(n\right)\mid\lcm\left[n\right]^{3}p_{n}\left(m\right)$
for all $n\geq0,m\geq1$ is similar, but requires a bit more computations.
Assume that claim for $n-1$, we prove it for $n$.

The numerators from part \enuref{reciprocal-polynomials} in \claimref{Additive_form}
\[
q_{m}\left(1\right)p_{n}\left(m+1\right)+p_{m}\left(1\right)q_{n}\left(m+1\right)=q_{n}\left(1\right)p_{m}\left(n+1\right)+p_{n}\left(1\right)q_{m}\left(n+1\right)
\]
can be rewritten as 
\[
\frac{\lcm\left[n\right]^{3}p_{n}\left(m+1\right)}{q_{n}\left(1\right)}=\overbrace{\frac{\lcm\left[n\right]^{3}p_{m}\left(n+1\right)}{q_{m}\left(1\right)}}^{\left(1\right)}+\overbrace{\frac{\lcm\left[n\right]^{3}p_{n}\left(1\right)}{q_{n}\left(1\right)}}^{\left(2\right)}\cdot\overbrace{\frac{q_{m}\left(n+1\right)}{q_{m}\left(1\right)}}^{\left(3\right)}-\overbrace{\frac{\lcm\left[n\right]^{3}p_{m}\left(1\right)}{q_{m}\left(1\right)}}^{\left(4\right)}\cdot\overbrace{\frac{q_{n}\left(m+1\right)}{q_{n}\left(1\right)}}^{\left(5\right)}.
\]
To show that the expression on the left is an integer for $\left(n,m\right)$
when $0\leq m\leq n-1$, it is enough to show that $\left(1\right)-\left(5\right)$
on the right are integers. 
\begin{itemize}
\item Expressions $\left(2\right)$ and $\left(4\right)$ are on the first
row of the matrix field, and we saw in the beginning of the proof
that it is an integer (and we use the fact that $\lcm\left[m\right]\mid\lcm\left[n\right]$).
\item Expressions $\left(3\right)$ and $\left(5\right)$ follows from the
claim about the denominators (which is independent of this proof about
the numerators).
\item Expressions $\left(1\right)$ is true by the induction hypothesis.
\end{itemize}
To conclude, we just saw that $q_{1}\left(n\right)\mid\lcm\left[n\right]^{3}p_{n}\left(m\right)$
is true for $\left(n,m\right)$ when $1\leq m\leq n$.

Using part \enuref{increase-m} in \claimref{dual-field-identities}
with $n=m$ for the numerators we get 
\begin{align*}
p_{n}\left(n+1\right) & =\frac{1}{\left(f\bar{f}\right)\left(0,n\right)}\left[f\left(0,n\right)\left(\left(f\bar{f}\right)\left(n,0\right)p_{n-1}\left(n\right)+f\left(n,n\right)p_{n}\left(n\right)\right)-\left(\left(f\bar{f}\right)\left(n,0\right)q_{n-1}\left(n\right)+f\left(n,n\right)q_{n}\left(n\right)\right)\right]\\
 & =\left(-n^{3}p_{n-1}\left(n\right)+6p_{n}\left(n\right)\right)+\left(q_{n-1}\left(n\right)-\frac{6}{n^{3}}q_{n}\left(n\right)\right).
\end{align*}

It follows that 
\[
\frac{\lcm\left[n\right]^{3}p_{n}\left(n+1\right)}{n!^{3}}=\left(-\frac{\lcm\left[n\right]^{3}p_{n-1}\left(n\right)}{\left(n-1\right)!^{3}}+6\frac{\lcm\left[n\right]^{3}p_{n}\left(n\right)}{n!^{3}}\right)+\frac{\lcm\left[n\right]^{3}}{q_{n}\left(1\right)}\left(q_{n-1}\left(n\right)-\frac{6}{n^{3}}q_{n}\left(n\right)\right).
\]
The summand on the right is an integer from the claim about the denominators,
and the left summand is an integer, because we already proved the
claim for the pairs $\left(n-1,n\right)$ and $\left(n,n\right)$.
Thus, we conclude that the claim holds for $\left(n,n+1\right)$ ,
and all together we have seen that $\left(\frac{n!}{\lcm\left[n\right]}\right)^{3}\mid p_{n}\left(m\right)$
for $1\leq m\leq n+1$. The same trick as with the denominators show
that $p_{n}\left(m\right)=p_{n}\left(1-m\right)$, so the claim is
true for $-n\leq m\leq n+1$, and using \lemref{binomial-polynomials}
again we conclude that it is true for all $m$, thus finishing the
proof for the induction step, and therefore the original claim.
\end{proof}
\newpage With this results, we can lower bound the greatest common
divisor of $P\left(n,m+1\right)$ and $Q\left(n,m+1\right)$.
\begin{cor}
\label{cor:factorial-reduction}For all $n,m\geq0$ we have that 
\[
q_{m}\left(1\right)q_{n}\left(1\right)\mid Q\left(n,m+1\right),
\]
\begin{align*}
\frac{q_{m}\left(1\right)q_{n}\left(1\right)}{\lcm\left[\max\left(m,n\right)\right]^{3}} & \mid\gcd\left(P\left(n,m+1\right),Q\left(n,m+1\right)\right).
\end{align*}
In particular for $n=m$ we get that 
\begin{align*}
\left(\frac{n!}{\lcm\left[n\right]}\cdot n!\right)^{3}=\frac{\left(q_{n}\left(1\right)\right)^{2}}{\lcm\left[n\right]^{3}} & \mid\gcd\left(P\left(n,n+1\right),Q\left(n,n+1\right)\right).
\end{align*}
\end{cor}

\begin{proof}
Using the presentation from \claimref{Additive_form}
\[
\left(\begin{smallmatrix}P\left(n,m+1\right)\\
Q\left(n,m+1\right)
\end{smallmatrix}\right)=\left(\begin{smallmatrix}q_{m}\left(1\right) & p_{m}\left(1\right)\\
0 & q_{m}\left(1\right)
\end{smallmatrix}\right)\left(\begin{smallmatrix}p_{n}\left(m+1\right)\\
q_{n}\left(m+1\right)
\end{smallmatrix}\right),
\]
and \lemref{q-n-lemma} we get that 
\begin{align*}
q_{m}\left(1\right)q_{n}\left(1\right) & \mid q_{m}\left(1\right)q_{n}\left(m+1\right)=Q\left(n,m+1\right)\\
\frac{q_{m}\left(1\right)q_{n}\left(1\right)}{\lcm\left[\max\left(m,n\right)\right]^{3}} & \mid q_{m}\left(1\right)p_{n}\left(m+1\right)+p_{m}\left(1\right)q_{n}\left(m+1\right)=P\left(n,m+1\right).
\end{align*}
\end{proof}
This factorial reduction property, will help us in the end to show
that $\zeta\left(3\right)$ is irrational, but it can also be used
to show more general properties of the matrix field, as follows.
\begin{thm}
\label{thm:all-directions-converge}Let $n_{i},m_{i}\geq1$ be any
sequence such that $\max\left(n_{i},m_{i}\right)\to\infty$. Then
$\frac{P\left(n_{i},m_{i}\right)}{Q\left(n_{i},m_{i}\right)}\to\zeta\left(3\right)$.
\end{thm}

This is of course a generalization of \thmref{limit_on_each_line}
for this specific $\zeta\left(3\right)$ matrix field case. Before
we turn to prove it, lets see what we can say on these limits for
fixed lines.
\begin{cor}
\label{cor:zeta_3_limit_on_lines}For any $m\geq1$, the limit for
the $Y=m$ line is$\limfi n\frac{p_{n}\left(m\right)}{q_{n}\left(m\right)}=\sum_{m}^{\infty}\frac{1}{k^{3}}.$
\end{cor}

\begin{proof}
Using the notation $\left(\begin{smallmatrix}P\left(n,m-1\right)\\
Q\left(n,m-1\right)
\end{smallmatrix}\right)=\left(\begin{smallmatrix}q_{m-1}\left(1\right) & p_{m-1}\left(1\right)\\
0 & q_{m-1}\left(1\right)
\end{smallmatrix}\right)\left(\begin{smallmatrix}p_{n}\left(m\right)\\
q_{n}\left(m\right)
\end{smallmatrix}\right)$, we get that 
\[
\frac{P\left(n,m-1\right)}{Q\left(n,m-1\right)}=\frac{p_{n}\left(m\right)}{q_{n}\left(m\right)}+\frac{p_{m-1}\left(1\right)}{q_{m-1}\left(1\right)}
\]
By \thmref{all-directions-converge} we know that $\limfi n\frac{P\left(n,m-1\right)}{Q\left(n,m-1\right)}=\zeta\left(3\right)$,
and we have already seen in \eqref{zeta_3_bottom_line} that $\frac{p_{m-1}\left(1\right)}{q_{m-1}\left(1\right)}=\sum_{1}^{m-1}\frac{1}{k^{3}}$,
so together we get that $\limfi n\frac{p_{n}\left(m\right)}{q_{n}\left(m\right)}=\sum_{m}^{\infty}\frac{1}{k^{3}}$.
\end{proof}
And now for the proof of the theorem.
\begin{proof}[Proof of \thmref{all-directions-converge}]
\textbf{\uline{The \mbox{$m_{i}$} bounded case}}\textbf{:}

We can assume here that $m_{i}$ is fixed, so this case is simply
an application of \thmref{limit_on_each_line}, for which we need
to show that $\left|\prod_{k=1}^{n}\frac{\bar{f}\left(k,0\right)}{f\left(k,0\right)}\right|=o\left(n^{d}\right)$
for $d=\deg_{y}\left(f\left(y,x\right)+\bar{f}\left(y,x+1\right)\right)$.
Since in our case $\left|\frac{\bar{f}\left(k,0\right)}{f\left(k,0\right)}\right|=1$
and $d=3$, this is clearly true. It follows that $\limfi n\frac{P\left(n,m\right)}{Q\left(n,m\right)}=L_{m}$
is constant in $m$ and since $\frac{P\left(n,1\right)}{Q\left(n,1\right)}=\frac{p_{n}\left(1\right)}{q_{n}\left(1\right)}=\sum_{1}^{n}\frac{1}{k^{3}}$,
the limt must be $\zeta\left(3\right)$.

\medskip{}

\newpage{}

\textbf{\uline{The \mbox{$m_{i}$} unbounded case}}\textbf{:} Here
we will use the second presentation of $P$ and $Q$, namely
\[
\left(\begin{smallmatrix}P\left(n,m\right)\\
Q\left(n,m\right)
\end{smallmatrix}\right)=\left(\begin{smallmatrix}q_{m}\left(1\right) & p_{m}\left(1\right)\\
0 & q_{m}\left(1\right)
\end{smallmatrix}\right)\left(\begin{smallmatrix}p_{n}\left(m+1\right)\\
q_{n}\left(m+1\right)
\end{smallmatrix}\right)
\]
and therefore
\[
\frac{P\left(n,m\right)}{Q\left(n,m\right)}=\frac{p_{n}\left(m+1\right)}{q_{n}\left(m+1\right)}+\frac{p_{m}\left(1\right)}{q_{m}\left(1\right)}.
\]

Using the intuition from \corref{zeta_3_limit_on_lines}, we expect
$\limfi m\frac{p_{n}\left(m+1\right)}{q_{n}\left(m+1\right)}=\sum_{m+1}^{\infty}\frac{1}{k^{3}}$
for any fixed $m$. As this is the tail of a convergent sum, as $m\to\infty$,
the limit goes to zero. If we can show that this somehow implies that
$\frac{p_{n}\left(m+1\right)}{q_{n}\left(m+1\right)}\overset{m\to\infty}{\longrightarrow}0$
uniformly in $n$, then $\limfi i\frac{P\left(n_{i},m_{i}\right)}{Q\left(n_{i},m_{i}\right)}=\limfi i\frac{p_{m_{i}}\left(1\right)}{q_{m_{i}}\left(1\right)}=\limfi i\sum_{1}^{m_{i}}\frac{1}{k^{3}}=\zeta\left(3\right)$.

More formally, fix some $\varepsilon>0$, and we need to show that
for all $m$ large enough $\left|\frac{p_{n}\left(m\right)}{q_{n}\left(m\right)}\right|\leq\frac{\varepsilon}{2}$
independent of $n$.  

Recall that 
\[
\left(\begin{smallmatrix}p_{n-1}\left(y\right) & p_{n}\left(y\right)\\
q_{n-1}\left(y\right) & q_{n}\left(y\right)
\end{smallmatrix}\right)D_{b_{X}\left(n\right)}=\prod_{0}^{n-1}M_{X}\left(k,y\right),
\]
so taking the determinant, we get that 
\[
\left(p_{n-1}\left(y\right)q_{n}\left(y\right)-p_{n}\left(y\right)q_{n-1}\left(y\right)\right)b\left(n\right)=\prod_{1}^{n}\left(-b\left(k\right)\right),
\]
which we can rewrite as
\[
\frac{p_{n}\left(y\right)}{q_{n}\left(y\right)}-\frac{p_{n-1}\left(y\right)}{q_{n-1}\left(y\right)}=\frac{\left(-1\right)^{n-1}\prod_{1}^{n-1}b\left(k\right)}{q_{n-1}\left(y\right)q_{n}\left(y\right)}.
\]

Using the fact that $p_{0}\left(y\right)=0$ , $q_{0}\left(y\right)=1$
we get the upper bound 
\[
\left|\frac{p_{n}\left(y\right)}{q_{n}\left(y\right)}\right|=\left|\sum_{j=1}^{n}\frac{\left(-1\right)^{j}\prod_{1}^{j-1}b\left(k\right)}{q_{j-1}\left(y\right)q_{j}\left(y\right)}\right|\leq\sum_{j=1}^{\infty}\left|\frac{\prod_{1}^{j-1}b\left(k\right)}{q_{j-1}\left(y\right)q_{j}\left(y\right)}\right|.
\]
By \lemref{q-n-lemma} we have that $\prod_{1}^{n-1}b\left(k\right)=\left(n-1\right)!^{6}=q_{n-1}\left(1\right)^{2}$,
and also $q_{n-1}\left(1\right)\mid q_{n-1}\left(m\right)$ and $q_{n-1}\left(1\right)n^{3}=q_{n}\left(n\right)\mid q_{n}\left(m\right)$
so that 
\[
\left|\frac{\prod_{1}^{j-1}b\left(k\right)}{q_{j-1}\left(m\right)q_{j}\left(m\right)}\right|=\left|\frac{q_{j-1}\left(1\right)}{q_{j-1}\left(m\right)}\cdot\frac{q_{j-1}\left(1\right)j^{3}}{q_{j}\left(m\right)}\cdot\frac{1}{j^{3}}\right|\leq\frac{1}{j^{3}}.
\]
This already shows that $\left|\frac{p_{n}\left(m\right)}{q_{n}\left(m\right)}\right|\leq\sum_{j=1}^{\infty}\frac{1}{j^{3}}$,
which is of course not enough, as instead of the tail, we got the
full sum. To solve this, we note that each one of the $q_{j}\left(y\right)$
for fixed $j\geq1$ are nonconstant polynomials of $y$ (of degree
$2j$) so that $\limfi y\left|\frac{q_{j-1}\left(1\right)}{q_{j-1}\left(y\right)}\cdot\frac{q_{j-1}\left(1\right)}{q_{j}\left(y\right)}\right|=0$.
Fixing $\varepsilon>0$ and $N>0$, we can find $M=M_{\varepsilon,N}$
large enough such that $\left|\frac{q_{j-1}\left(1\right)}{q_{j-1}\left(y\right)}\cdot\frac{q_{j-1}\left(1\right)}{q_{j}\left(y\right)}\right|<\frac{\varepsilon}{N}$
for all $y>M$ and $1\leq j<N$. In particular, for any integer $m>M$
we have that 
\[
\left|\frac{p_{n}\left(m\right)}{q_{n}\left(m\right)}\right|\leq\sum_{j=1}^{\infty}\left|\frac{q_{j-1}\left(1\right)}{q_{j-1}\left(m\right)}\cdot\frac{q_{j-1}\left(1\right)}{q_{j}\left(m\right)}\right|\leq\frac{\varepsilon}{N}N+\sum_{j=N}^{\infty}\left|\frac{q_{j-1}\left(1\right)}{q_{j-1}\left(m\right)}\cdot\frac{q_{j-1}\left(1\right)}{q_{j}\left(m\right)}\right|\leq\varepsilon+\sum_{j=N}^{\infty}\frac{1}{j^{3}}.
\]
Since $\sum_{1}^{\infty}\frac{1}{j^{3}}<\infty$ converges, we can
find $N$ large enough so that $\sum_{N}^{\infty}\frac{1}{j^{3}}\leq\varepsilon$
also, so together we get that for all $m$ big enough (independent
of $n$) we have $\left|\frac{p_{n}\left(m\right)}{q_{n}\left(m\right)}\right|\leq2\varepsilon$
which is what we wanted to prove.
\end{proof}
\newpage Finally, we combine all of the results to show that $\zeta\left(3\right)$
is irrational.
\begin{thm}
\label{thm:zeta-3-irrational}The number $\zeta\left(3\right)$ is
irrational.
\end{thm}

\begin{proof}
Consider the diagonal direction on the $\zeta\left(3\right)$ matrix
field where $m=n+1$. From \thmref{all-directions-converge} we have
that 
\[
\limfi n\frac{P\left(n,n+1\right)}{Q\left(n,n+1\right)}=\zeta\left(3\right).
\]

\textbf{\uline{The main idea:}}

Let us denote $Q_{n}=Q\left(n,n+1\right),\;P_{n}=P\left(n,n+1\right)$
so that $\limfi n\frac{P_{n}}{Q_{n}}=\zeta\left(3\right)$. Recall
that in \thmref{improved-rationality-test} we showed that if $\frac{P_{n}}{Q_{n}}$
is not eventually constant and $\left|Q_{n}\zeta\left(3\right)-P_{n}\right|=o\left(gcd\left(P_{n},Q_{n}\right)\right)$,
then $\zeta\left(3\right)$ is irrational. We will begin by showing
that the diagonal is also a polynomial continued fraction in disguise,
and then use it to approximate the errors and denominators.

Setting $\lambda_{+}=\left(1+\sqrt{2}\right)^{4}$, we will first
show that given any $\varepsilon>0$ we have:
\[
\left.\begin{array}{c}
\left|\zeta\left(3\right)-\frac{P_{n}}{Q_{n}}\right|=O\left(\frac{1}{\left(\lambda_{+}-\varepsilon\right)^{2n}}\right)\\
Q_{n}=O\left(\left(n!\right)^{6}\left(\lambda_{+}+\varepsilon\right)^{n}\right)
\end{array}\right\} \;\Rightarrow\;\left|Q_{n}\zeta\left(3\right)-P_{n}\right|=O\left(\frac{\left(n!\right)^{6}\left(\lambda_{+}+\varepsilon\right)^{n}}{\left(\lambda_{+}-\varepsilon\right)^{2n}}\right)\sim\frac{\left(n!\right)^{6}}{\lambda_{+}^{n}}.
\]
We then use the well known result that $\lcm\left[n\right]=O\left(\left(e+\varepsilon\right)^{n}\right)$
(it follows from the prime number theorem, see \cite{apostol_introduction_1998})
together with \corref{factorial-reduction} to get that 
\[
\frac{\left(n!\right)^{6}}{e^{3n}}\sim\left(\frac{n!}{\lcm\left[n\right]}\cdot n!\right)^{3}\mid\gcd\left(P_{n},Q_{n}\right).
\]
Finally, since $20.08\sim e^{3}<\lambda_{+}=\left(1+\sqrt{2}\right)^{4}\sim33.97$,
we could use the irrationality test and show that $\zeta\left(3\right)$
is irrational.

\medskip{}

\textbf{\uline{Step 1: Find recursion relation for \mbox{$Q_{n}$}:}}

With this main idea, we are left to find the growth rate of $Q_{n}$
and how fast $\left|\zeta\left(3\right)-\frac{P_{n}}{Q_{n}}\right|$
goes to zero. 

Using the conservativeness condition on the matrix field, we get that
\begin{align*}
\left(\begin{smallmatrix}P_{n}\\
Q_{n}
\end{smallmatrix}\right)=\left(\begin{smallmatrix}P\left(n,n+1\right)\\
Q\left(n,n+1\right)
\end{smallmatrix}\right) & =\left[\prod_{k=1}^{n}M_{Y}\left(0,k\right)\right]\left[\prod_{k=0}^{n-1}M_{X}\left(k,n+1\right)\right]e_{2}=\left[\prod_{k=1}^{n}M_{X}\left(k-1,k\right)M_{Y}\left(k,k\right)\right]e_{2},
\end{align*}
where 
\begin{align*}
M_{X}\left(k-1,k\right)M_{Y}\left(k,k\right) & =\left(\begin{smallmatrix}0 & 1\\
b\left(k\right) & f\left(k-1,k\right)-\bar{f}\left(k,k\right)
\end{smallmatrix}\right)\left(\begin{smallmatrix}\bar{f}\left(k,k\right) & 1\\
b\left(k\right) & f\left(k,k\right)
\end{smallmatrix}\right)\\
 & =\left(\begin{smallmatrix}b\left(k\right) & f\left(k,k\right)\\
f\left(k-1,k\right)b\left(k\right)\; & \left(f\bar{f}\right)\left(k,0\right)+f\left(k,k\right)f\left(k-1,k\right)-\left(f\bar{f}\right)\left(k,k\right)
\end{smallmatrix}\right)\\
 & =\left(\begin{smallmatrix}-k^{6} & 6k^{3}\\
-f\left(k-1,k\right)k^{6}\; & 6k^{3}f\left(k-1,k\right)-k^{6}
\end{smallmatrix}\right)=k^{3}\left(\begin{smallmatrix}-k^{3} & 6\\
-f\left(k-1,k\right)k^{3}\; & 6f\left(k-1,k\right)-k^{3}
\end{smallmatrix}\right).
\end{align*}

Fortunately, the last matrix is also a polynomial continued matrix
in disguise (see \cite{david_euler_2023} for the general process
of converting a polynomial matrix to continued fraction form). Indeed,
setting $U\left(k\right)=\left(\begin{smallmatrix}0 & 6\\
1 & \left(6f\left(k-1,k\right)-k^{3}\right)\;
\end{smallmatrix}\right)$ we get 
\[
M\left(k\right):=U\left(k\right)^{-1}\left(\begin{smallmatrix}-k^{3} & 6\\
-f\left(k-1,k\right)k^{3}\; & 6f\left(k-1,k\right)-k^{3}
\end{smallmatrix}\right)U\left(k+1\right)=\begin{pmatrix}0 & -k^{6}\\
1 & 6f\left(k,k+1\right)-k^{3}-\left(1+k\right)^{3}
\end{pmatrix},
\]
and therefore 
\begin{align*}
\frac{1}{\left(n!\right)^{3}}\left(\begin{smallmatrix}P_{n}\\
Q_{n}
\end{smallmatrix}\right) & =U\left(1\right)\left[\prod_{k=1}^{n}M\left(k\right)\right]U\left(n+1\right)^{-1}e_{2}=\left(\begin{smallmatrix}0 & 6\\
1 & 5
\end{smallmatrix}\right)\left[\prod_{k=1}^{n}M\left(k\right)\right]e_{1}=\left(\begin{smallmatrix}0 & 6\\
1 & 5
\end{smallmatrix}\right)\left[\prod_{k=1}^{n-1}M\left(k\right)\right]e_{2}.
\end{align*}

Setting as usuall $\begin{pmatrix}Q'_{n}\\
P'_{n}
\end{pmatrix}=\prod_{k=1}^{n-1}M\left(k\right)e_{2}$, both $P_{n}',Q_{n}'$ and therefore $u_{n}:=\frac{Q_{n}}{\left(n!\right)^{3}}=P_{n}'+5Q_{n}'$
satisfy the same reccurence

\begin{align*}
u_{n+1} & =u_{n}\left(6f\left(n,n+1\right)-\left(n+1\right)^{3}-n^{3}\right)-u_{n-1}n^{6},
\end{align*}
where $u_{0}=\frac{Q_{0}}{0!^{3}}=1$ and $u_{1}=\frac{Q_{1}}{1!^{3}}=6f\left(0,1\right)-1=5$.
Denote
\begin{align*}
F\left(n\right) & =6f\left(n,n+1\right)-\left(n+1\right)^{3}-n^{3}=34n^{3}+51n^{2}+27n+5,
\end{align*}
so the recurrence can be written as $u_{n+1}=F\left(n\right)u_{n}-n^{6}u_{n-1}$.
In particular, we already obtain that 
\[
\zeta\left(3\right)=\limfi n\frac{P_{n}}{Q_{n}}=\left(\begin{smallmatrix}0 & 6\\
1 & 5
\end{smallmatrix}\right)\left[\prod_{1}^{n}\left(\begin{smallmatrix}0 & -k^{6}\\
1 & F\left(k\right)
\end{smallmatrix}\right)\right]\left(0\right)=\left(\begin{smallmatrix}0 & 6\\
1 & 5
\end{smallmatrix}\right)\left(\KK_{1}^{\infty}\frac{-k^{6}}{F\left(k\right)}\right),
\]
or alternatively our diagonal (and Ap\'ery's original continued fraction)
gives
\[
\frac{6}{\zeta\left(3\right)}-5=\KK_{1}^{\infty}\frac{-k^{6}}{F\left(k\right)}.
\]
\medskip{}

\textbf{\uline{Step 2: Analyze the recurrence to find the growth
rate of \mbox{$Q_{n}$}:}}

By \corref{factorial-reduction} we know that $\left(n!\right)^{6}\mid Q_{n}$,
so that $v_{n}=\frac{Q_{n}}{\left(n!\right)^{6}}=\frac{u_{n}}{\left(n!\right)^{3}}$
are integers which satisfy
\[
v_{n+1}\left(n+1\right)^{3}=F\left(n\right)v_{n}-n^{3}v_{n-1},
\]
where $v_{0}=\frac{Q_{0}}{0!^{6}}=1$ and $v_{1}=\frac{Q_{1}}{1!^{6}}=5$.
Equivalently, we can write 
\[
v_{n+1}=\frac{F\left(n\right)}{n^{3}}v_{n}-\frac{n^{3}}{\left(1+n\right)^{3}}v_{n-1}.
\]
Taking the limit only for the coefficients, we get the ``limit''
recurrence $v_{n+1}'=34v_{n}'-v_{n-1}'$. This corresponds to the
quadratic equation $x^{2}-34x+1=0$ with the roots 
\[
\lambda_{\pm}=\frac{34\pm\sqrt{1156-4}}{2}=\frac{34\pm24\sqrt{2}}{2}=17\pm12\sqrt{2}=\left(1\pm\sqrt{2}\right)^{4},
\]
so a standard computation shows that $\frac{v_{n}'}{v_{n-1}'}\to\lambda_{+}=\left(1+\sqrt{2}\right)^{4}$.
In the original recurrence with the nonconstant coefficients, the
same holds, but needs a bit more explanation. As with the standard
recurrence with constant coefficients, we expect the general solution
to behave like $v_{n}\sim\lambda_{+}^{n}$, though there is a specific
starting position for which $v_{n}\sim\lambda_{-}^{n}$. Since $\lambda_{-}=\left(1-\sqrt{2}\right)^{4}\sim0.03$
, this is highly unlikely to happen, since we deal with integer values.
More sepcifically, the first few elements in $v_{i}$ are $1,5,73,1445,33001,...$
which is an increasing sequence of positive integers, and since $\frac{F\left(n\right)}{n^{3}}\geq11$
for $n\geq3$, it is not hard to show by induction that 
\[
v_{n+1}=\frac{F\left(n\right)}{n^{3}}v_{n}-\frac{n^{3}}{\left(1+n\right)^{3}}v_{n-1}\geq11v_{n}-v_{n-1}\geq10v_{n},
\]
so at least we get that $v_{n}\geq10^{n}$ grows much faster than
the very special case of $\lambda_{-}^{n}$. This is enough to show
that for every $\varepsilon>0$ and for any $n$ large enough, we
have
\[
\left(\lambda_{+}-\varepsilon\right)^{n}\leq v_{n}\leq\left(\lambda_{+}+\varepsilon\right)^{n}.
\]
The rest of the details are standard computations, and we leave it
to the reader.

\medskip{}

\newpage{}

\textbf{\uline{Step 3: Analyze the approximation error:}}

The sequence of $\frac{P_{n}}{\left(n!\right)^{3}},u_{n}=\frac{Q_{n}}{\left(n!\right)^{3}}$
are the numerators and denominators of the continued fraction $\KK_{1}^{\infty}\frac{-n^{6}}{F\left(n\right)}$.
Using \claimref{upper-bound} we get that for all $n$ large enough
\[
\left|\zeta\left(3\right)-\frac{P_{n}}{Q_{n}}\right|\leq\sum_{k=n}^{\infty}\frac{\left(k!\right)^{6}}{\left|u_{k}u_{k+1}\right|}=\sum_{k=n}^{\infty}\frac{1}{\left(k+1\right)^{3}\left|v_{k}v_{k+1}\right|}\leq\sum_{n}^{\infty}\frac{1}{\left(\lambda_{+}-\varepsilon\right)^{2k+1}}=O\left(\frac{1}{\left(\lambda_{+}-\varepsilon\right)^{2n}}\right).
\]

These are the growth rate for $Q_{n}=\left(n!\right)^{6}v_{n}$ and
the error for $\left|\zeta\left(3\right)-\frac{P_{n}}{Q_{n}}\right|$
that we needed in the beginning, thus completing the proof.
\end{proof}

\section{\label{sec:On-future-fractions}On future polynomial continued fractions}

The main goal of this paper was to introduce the conservative matrix
field, and as an application reprove Ap\'ery's result about the irrationality
of $\zeta\left(3\right)$. As can be seen in \secref{The-z3-case},
the final proof was very specific to the matrix field of $\zeta\left(3\right)$,
which has several nice properties, and doesn't hold in general. However
it might be possible (as numerical computations suggest) that some
of the results hold in a more general setting.

While this irrationality result is already interesting by itself,
the conservative matrix field object also seems to have many interesting
properties. Among others, it is a natural generalization of quadratic
equations, and it involves a bit of noncommutative cohomology theory
in the form of cocycles and coboundaries.

So far, the conservative matrix fields that we managed to find where
$f,\bar{f}$ are conjugate polynomials of degree 4 or more seem to
always be degenerate, namely $a\left(x,y\right)=f\left(x,y\right)-\bar{f}\left(x+1,y\right)$
doesn't depend on $y$. This might be related to the fact that we
work over $2\times2$ matrices, which might bound the possible matrix
fields. Whether this is the case or not, this leads to several possible
interesting generalizations of this theory, which are standard in
the theory of continued fractions and in number theory in general.
\begin{enumerate}
\item While many of the results mentioned in this paper are true for general
continued fractions over $\CC$ (and even other fields), the irrationality
of $\zeta\left(3\right)$ relied heavily on the fact that the defining
polynomials $f,\bar{f}$ were in $\ZZ\left[x,y\right]$. This leads
naturally to the question of what happens when we use other integer
rings in algebraic extensions, e.g. $\ZZ\left[i\right]$ or $\ZZ\left[\sqrt{2}\right]$.
Both in the $\zeta\left(2\right)$ and $\zeta\left(3\right)$ matrix
fields cases we can find in the background algebraic numbers of degree
$2$ (namely $1+i$ and $\zeta_{3}=e^{\frac{2\pi}{3}i}$ respectively).
This type of field extension, with the right definition of generalized
continued fraction might add many more interesting examples.
\item In the proof of the irrationality of $\zeta\left(3\right)$ we had
two main results that we needed to show. One was to find the error
rate and how fast it converges to zero, and the second was to find
$gcd\left(P_{n},Q_{n}\right)$ and hope that it grows to infinity
fast enough. As it is usually the case in number theoretic problems,
the first result lives in the standard Euclidean geometry, where we
needed to show that some sequence goes to zero in the $\left|\cdot\right|_{\infty}$
norm, and the second result can be seen as showing that the $p$-adic
norms of $\left|gcd\left(P_{n},Q_{n}\right)\right|_{p}$ all go to
zero as well. This suggests a more general approach where the matrix
field lives over the Adeles, and the convergence in the real and $p$-adic
places together prove irrationality.
\item All the results in this paper were for $2\times2$ matrices, and a
natural generalization would be going to a higher dimension matrices.
There are many suggestions for what should be the generalization of
continued fractions to higher dimensions, however probably one of
the best approaches is to change the language all together from continued
fractions to lattices in $\RR^{d}$. The subject of lattices is well
studied in the literature with many connections to other subjects.
With this approach, the question should be what is the right way to
formulate the results about general continued fraction as results
on lattices, and what can we say in higher dimension.
\end{enumerate}
These three types of generalization of changing the field, the norm,
or the dimension, can also be combined. Of course, there are more
tools available already in ``standard'' $2\times2$ matrices over
the integers to study polynomial continued fraction. However, it seems
that the conservative matrix field holds some interesting structure
which might reveal itself to be very useful not only to prove results
about continued fractions, but to other subjects as well.

\newpage{}

\part{Appendix}

\appendix

\section{\label{app:integer-values}The algebra of integer valued polynomial}

One of the tools we need along the way were rational polynomial $f\in\QQ\left[x\right]$
where $f\left(\ZZ\right)\subseteq\ZZ$, which are called \textbf{integer
valued polynomials}. Of course, if we can write $f\left(x\right)\in\ZZ\left[x\right]$,
then $f$ is integer valued, but the other direction is not true.
For example, $2\mid g\left(n\right)=n\left(n+1\right)$ for every
integer $n$, so that $f\left(x\right)=\frac{x\left(x+1\right)}{2}$
is an integer valued polynomial which is not in $\ZZ\left[x\right]$.
There are many more such polynomials, which we write as follows:
\begin{defn}
For $n\in\NN$ we define the polynomial $\binom{x}{n}=\prod_{0}^{n-1}\frac{\left(x-i\right)}{i+1}=\frac{x\cdot\left(x-1\right)\cdots\left(x-n+1\right)}{n!}$.
This is a polynomial of degree $n$ in $\QQ\left[x\right]$, so that
$\left\{ \binom{x}{n}\right\} _{0}^{\infty}$ is a $\QQ$-basis for
$\QQ\left[x\right]$. In particular, for a nonnegative integers $x$,
we simply get the binomials.
\end{defn}

The integer valued polynomials were fully described by Pólya in \cite{polya_uber_1915},
and were shown to contain exactly the integer combinations of the
$\binom{x}{n}$ above. For the ease of the reader, we add the proof
for this result here.
\begin{lem}
For every integer $m$, we have that $\binom{m}{n}\in\ZZ$. 
\end{lem}

\begin{proof}
For $m\geq n$, these are just the binomial coefficients $\frac{m!}{n!\left(m-n\right)!}$,
which are integers, and for $0\leq m<n$ by definition $\binom{m}{n}=0$.
Finally, for negative $m<0$ we get that $\binom{m}{n}=\left(-1\right)^{n}\binom{n-(m+1)}{n}$
where $n-\left(m-1\right)\geq n$ which is again a binomial coefficient
(up to a sign) and therefore an integer.
\end{proof}
In the following, by writing $d\mid q$ for $d\in\ZZ,q\in\QQ$, we
mean that $q$ has to be an integer divisible by $d$.
\begin{lem}
\label{lem:binomial-polynomials}Given a general polynomial $f\left(x\right)=\sum_{0}^{d}a_{n}\binom{x}{n}\in\QQ\left[x\right]$
and an integer $k$ the following are equivalent:
\begin{enumerate}
\item For all $0\leq n\leq d$ we have $k\mid a_{n}$ ,
\item For all $m\in\ZZ$ we have $k\mid f\left(m\right)$ ,
\item For $m=0,1,...,d$ , we have $k\mid f\left(m\right)$ ,
\item There exists $m_{0}$ such that $k\mid f\left(m_{0}+m\right)$ for
$m=0,1,...,d$, and
\end{enumerate}
\end{lem}

\begin{proof}
Note first that considering the polynomial $\frac{f\left(x\right)}{k}$
instead, it is enough to prove the lemma for $k=1$. Namely, we just
need to show that the coefficients\textbackslash evaluation are integers.
\begin{itemize}
\item $\left(1\right)\Rightarrow\left(2\right)$: follows from the fact
that $\binom{m}{n}$ are integers for all $m$.
\item $\left(2\right)\Rightarrow\left(3\right)$: is trivial.
\item $\left(3\right)\Rightarrow\left(1\right)$: Since $\binom{n}{n}=1$
and $\binom{m}{n}=0$ when $0\leq m<n$, it follows that for $0\leq m\leq d$
we have 
\[
f\left(m\right)=a_{m}+\sum_{n=0}^{m-1}a_{n}\binom{m}{n}\qquad\iff\qquad a_{m}=f\left(m\right)-\sum_{n=0}^{m-1}a_{n}\binom{m}{n}.
\]
So if $a_{n}\in\ZZ$ for $0\leq n<m$, then since $f\left(m\right)\in\ZZ$
by assumption and $\binom{m}{n}\in\ZZ$, we conclude that $a_{m}\in\ZZ$.
Thus, by induction we get that $a_{n}\in\ZZ$ for all $0\leq n\leq d$,
namely we get $\left(1\right)$.
\item $\left(2\right)\Rightarrow\left(4\right)$: is trivial.
\item $\left(4\right)\Rightarrow\left(2\right)$: If $f\left(m_{0}+m\right)\in\ZZ$
for $m=0,...,d$, then setting $g\left(m\right)=f\left(m_{0}+m\right)$
we see that $g\left(m\right)\in\ZZ$ for $m=0,...,d$. By the $\left(3\right)\Rightarrow\left(2\right)$
direction for the degree $d$ polynomial $g$ we get that $g\left(m\right)=f\left(m_{0}+m\right)\in\ZZ$
for all $m$, which is exactly condition $\left(2\right)$ for the
polynomial $f$ and we are done.
\end{itemize}
\end{proof}

\newpage{}

\bibliographystyle{plain}
\bibliography{apery}

\end{document}